\documentclass[letterpaper, 12pt]{amsart}

\usepackage{amsmath,amssymb,amsxtra,amsthm,hyperref}
\usepackage{tikz}
\usepackage{multicol}

\usepackage{enumitem} 
\usepackage{tensor} 
\usepackage{tikz-cd} 
\usepackage{upref}  
\usepackage{mathrsfs} 
\usepackage{chngcntr}
\usepackage[normalem]{ulem} 

\newcommand{\condref}[1] {\hyperref[cond.#1]{(#1)}}

\newcommand{\cA}{\mathcal{A}}
\newcommand{\cB}{\mathscr{B}}
\newcommand{\cC}{\mathscr{C}}

\newcommand{\cG}{\mathcal{G}}
\newcommand{\cH}{\mathcal{H}}

\newcommand{\cK}{\mathcal{K}}
\newcommand{\cL}{\mathcal{L}}
\newcommand{\cM}{\mathcal{M}}

\newcommand{\cU}{\mathcal{U}}

\newcommand{\ZS}{Zappa-Sz\'{e}p }
\newcommand{\etale}{\'{e}tale }

\newtheorem{theorem}{Theorem}[section]

\newtheorem{corollary}[theorem]{Corollary}
\newtheorem{proposition}[theorem]{Proposition}   
 
\newtheorem{lemma}[theorem]{Lemma}

\theoremstyle{definition}
\newtheorem{definition}[theorem]{Definition}

\newtheorem{example}[theorem]{Example}
\newtheorem{remark}[theorem]{Remark}

\numberwithin{equation}{section}

\newcommand{\norm}[1]{\left\| #1 \right\|}
\newcommand{\abs}[1]{\left| #1 \right|}
\newcommand{\inner}[2]{\left\langle #1 \mid #2 \right\rangle}

\newcommand{\Biginner}[2]{\Big\langle #1 \Bigm\vert #2 \Big\rangle}
\newcommand{\inv}{^{-1}}
\newcommand{\z}{^{(0)}}
\DeclareMathOperator{\supp}{supp}
\DeclareMathOperator{\End}{End}
\DeclareMathOperator{\Iso}{Iso}

\let\mscr\mathscr

\let\mc\mathcal
\let\mf\mathfrak

\begin{document}

\title{The Zappa-Sz\'{e}p product of a Fell bundle and a groupoid}

\author{Anna Duwenig}
\address{School of Mathematics and Applied Statistics, University of Wollongong, Wollongong, Australia}
\email{aduwenig@uow.edu.au}

\author{Boyu Li}
\address{Department of Mathematics and Statistics, University of Victoria, Victoria, B.C., Canada}
\email{boyuli@uvic.ca}
\date{\today}

\thanks{The second author was supported by a postdoctoral fellowship of Pacific Institute for the Mathematical Sciences.}


\begin{abstract} We define the Zappa-Sz\'{e}p product of a Fell bundle by a groupoid, which turns out to be a Fell bundle over the Zappa-Sz\'{e}p product of the underlying groupoids. Under certain assumptions, every Fell bundle over the Zappa-Sz\'{e}p product of groupoids arises in this manner. We then study the representation associated with the Zappa-Sz\'{e}p product Fell bundle and show its relation to covariant representations. Finally, we study the associated universal $C^*$-algebra, which turns out to be a $C^*$-blend, generalizing an earlier result about the Zappa-Sz\'{e}p product of groupoid $C^*$-algebras. In the case of discrete groups, the universal $C^*$-algebra of a Fell bundle embeds injectively inside the universal $C^*$-algebra of the Zappa-Sz\'{e}p product Fell bundle. 
\end{abstract}

\maketitle

\section{Introduction}

The \ZS product originated as a generalization of the semi-direct product of groups. For groups $G$ and $H$, in addition to encoding an $H$-action on $G$ in the semi-direct product, the \ZS product encodes a $G$-restriction map on $H$. This results in a two-way interaction between $G$ and $H$ in their \ZS product. 

The crossed product construction, in the realm of operator algebras, closely resembles that of a semi-direct product: given a $C^*$-algebra $A$ and a group $H$ acting on ${A}$ by automorphisms, one can define the algebraic crossed product ${A}\rtimes^{\operatorname{alg}}_\alpha H$ in a similar fashion to the semi-direct product. There are various ways to take the closure of the algebraic crossed product that could result in different $C^*$-algebras that have drawn much research interest. 

Naturally, one may wonder if we can similarly construct a \ZS product of a $C^*$-algebra. This is the main motivation behind this paper. To define a \ZS product of a $C^*$-algebra ${A}$, one must encode both an $H$-action on the $C^*$-algebra and an ${A}$-restriction on the group $H$. One possible approach is to put a $G$-grading on ${A}$, by dissembling ${A}$ into linearly independent subspaces $\{\cA_x\}_{x\in G}$ whose direct sum is dense in ${A}$ (for example, the notion of a graded $C^*$-algebra considered by Exel \cite{Exel1997}). With this approach, the elements in $\cA_x$ induce a restriction map on $H$ in a similar manner as the $G$-restriction map in a \ZS product. Algebraically, one can define a \ZS type product of the subspaces $\{\cA_x\}_{x\in G}$ and $H$. However, this approach faces a key challenge: it is difficult, sometimes even impossible to reassemble the individual pieces back to a $C^*$-algebra, and even if that is possible, the resulting $C^*$-algebra might have lost much of the information of the original $C^*$-algebra ${A}$. 

To avoid these difficulties, instead of decomposing a $C^*$-algebra ${A}$, we start directly with a collection of fibres $\{\cA_x\}$ that can be reasonably assembled into different $C^*$-algebras.  The approach to study the fibres instead of the $C^*$-algebra as a whole leads us to the notion of {{\em Fell bundles}.} 

Suppose $\cG$ and $\cH$ are \etale groupoids that have a \ZS product groupoid $\cG\bowtie \cH$,  which is known to be \etale as well \cite{BPRRW2017}. Given a Fell bundle $\cB$ over $\cG$, we first define the notion of a {\em $(\cG,\cH)$-compatible $\cH$-action} $\beta$ on the Fell bundle (Definition \ref{df.beta}) that allows us to construct a Zappa-Sz\'{e}p-type product $\cB\bowtie_\beta \cH$. This product turns out to be a Fell bundle over the \ZS product of the underlying groupoids $\cG\bowtie\cH$ (Theorem \ref{thm.bundle}). Conversely, certain Fell bundles over $\cG\bowtie \cH$ arise in this manner (Theorem \ref{thm.internalZS}).

We then study representations of the \ZS product Fell bundle. We define (Definition \ref{def:cov-rep}) the notion of {\em covariant representations} of the system $(\cB,\beta)$, in parallel to the classical notion of covariant representations of a $C^*$-dynamical system. We prove an integration theorem that every covariant representation gives rise to an $I$-norm decreasing $*$-representation of $\Gamma_c(\cG\bowtie\cH; \cB\bowtie_\beta \cH)$ (Theorem \ref{thm:representations-of-ZS-Fellbdl}). Conversely, under the assumption that $\cB$ is unital, we also prove a disintegration theorem that every nondegenerate $I$-norm decreasing $*$-representation of $\Gamma_c(\cG\bowtie\cH; \cB\bowtie_\beta \cH)$ is equivalent to the integrated form of some covariant representation (Theorem \ref{thm.disintegration}). 

Finally, we study the universal $C^*$-algebra associated with the \ZS product Fell bundle. We first prove that the $C^*$-algebra of the \ZS product Fell bundle $C^*(\cB\bowtie_\beta \cH)$ is a $C^*$-blend of $C^*(\cB)$ and the groupoid $C^*$-algebra $C^*(\cH)$ when $\cB$ is unital (Theorem \ref{thm.blend}). This generalizes an earlier result \cite{BPRRW2017} that the groupoid $C^*$-algebra of the \ZS product groupoid $C^*(\cG\bowtie \cH)$ is a $C^*$-blend of $C^*(\cG)$ and $C^*(\cH)$. Finally, in the case of discrete groups, we prove that any representation of the Fell bundle $\cB$ induces a covariant representation of the \ZS product Fell bundle (Lemma \ref{lm.M}), in analogy to the classical regular representation of $C^*$-dynamical systems (for example, \cite[Example 2.14]{Wil2007}). As a result, $C^*(\cB)$ embeds injectively inside $C^*(\cB\bowtie_\beta\Lambda)$ for $\Lambda$ a discrete group (Theorem \ref{thm.inj}).

In recent years, the \ZS product has attracted much attention in the study of operator algebras. For example, recent studies include the $C^*$-algebra of self-similar groups \cite{Nekrashevych2009}; the \ZS product of right LCM semigroups \cite{BRRW, Starling2015} and \etale groupoids \cite{BPRRW2017}; self-similar group actions on graphs \cite{EP2017} and on $k$-graphs \cite{LY2019, LY2019b}; self-similar groupoid actions on $k$-graphs \cite{ABRW2019}. Our hope is that this paper brings a new perspective into this line of research. 

\pagebreak[4]\section{Preliminary}

\subsection{Zappa-Sz\'{e}p Product of Groupoids} In group theory, the Zappa-Sz\'{e}p product provides a way to construct a group from certain interactions between two groups. It is a natural generalization of the semi-direct product $G\rtimes H$ of groups $G$ and $H$, which encodes an $H$-action on the group $G$ by defining the multiplication
\[(x,h)(y,k)=(x(h\cdot y), hk).\]
However, it is possible that $G$ also acts on $H$. This is known as the $G$-restriction map $(x,h)\mapsto h|_x$ in the \ZS product, and the multiplication in the \ZS product encodes this two-way action by setting
\[(x,h)(y,k)=(x(h\cdot y), h|_y k).\]

For other algebraic structures, one can often define an analogous version of their Zappa-Sz\'{e}p product: for example, the case of right LCM semigroups is considered in \cite{BRRW}. For two groupoids $\cG$ and $\cH$, one can define a similar notion of Zappa-Sz\'{e}p product when their unit spaces $\cG\z, \cH\z$ match \cite{AA2005}, which we will recall below. 
One may refer to \cite{PatersonGroupoidBook} for a detailed introduction of groupoids and to \cite{BPRRW2017} for a detailed discussion on their \ZS products. 

Let $\cG$ and $\cH$ be 
groupoids. For $x\in \cG$, define its source by $s_\cG(x)=x^{-1}x$ and its range by $r_\cG(x)=xx^{-1}$; similarly, define $s_{\cH}$ and $r_{\cH}$ for $\cH$. 
 We assume that $\cG$ and $\cH$ have the same unit space $\cG\z=\cH\z$,
 so that the ranges of the maps $s_\cG,s_\cH,r_\cG,r_\cH$ are all the same set. Assume further that there exist a continuous right-action of $\cG$ on $s_{\cH}\colon \cH \to\cG\z$, and a continuous left-action of $\cH$ on $r_{\cG}\colon \cG \to \cH\z$. Following \cite{BPRRW2017}, we denote these by
\begin{enumerate}[leftmargin=3cm]
    \item[$\cH\curvearrowright\cG$:\quad] 
    $\cH \tensor*[_{s_{\cH}}]{\times}{_{r_{\cG}}}\cG \to \cG$, $(h,x)\mapsto h\cdot x$, and 
    \item[$\cH\curvearrowleft\cG$:\quad] $\cH \tensor*[_{s_{\cH}}]{\times}{_{r_{\cG}}}\cG \to\cH$, $(h,x)\mapsto h|_x$,
\end{enumerate}
where we used the following notation for the fibered product:
\begin{equation}\label{eq:fibred-product}
    Y \tensor*[_{q}]{\times}{_{p}} X
    := 
    \bigl\{(y,x)\in Y\times X : q(y)=p(x)\bigr\} 
    \text{ whenever } 
    \begin{tikzcd}[column sep = tiny, row sep = small]
        Y \ar[dr, shorten <= -.5ex, "q"']& & \ar[dl, shorten <= -.5ex, "p"] X
        \\
        & Z & 
    \end{tikzcd}
    .
\end{equation}
We call $\cdot$ the {\em $\cH$-action map} and $\vert$ the {\em $\cG$-restriction map}. Recall the following properties, intrinsic to groupoid actions:
\begin{multicols}{2}
\begin{enumerate}[label=(ZS\arabic*)]
\item\label{cond.ZS1}\label{ZSH-action-associative} $(h_1h_2)\cdot x=h_1\cdot (h_2\cdot x)$,
\item\label{cond.ZS5}\label{ZSH-action-range} $r_\cG(h\cdot x)=r_\cH(h)$,
\item\label{cond.ZS8}\label{ZSH-action-unit} $r_\cG(x)\cdot x=x$,
\item\label{cond.ZS3}\label{ZSG-restriction-associative} $h|_{xy}=(h|_x)|_y$,
\item\label{cond.ZS6}\label{ZSG-restriction-source} $s_\cH(h|_x)=s_\cG(x)$,
\item\label{cond.ZS9}\label{ZSG-restriction-unit} $h|_{s_\cH(h)}=h$.
\end{enumerate}
\end{multicols}
We further assume that these actions
satisfy the following compatibility conditions:
\begin{enumerate}[label=(ZS\arabic*)]\setcounter{enumi}{6}
\item\label{cond.ZS7}\label{ZSaction-restriction-sH=rG} $s_\cG(h\cdot x)=r_\cH(h|_x)$,
\item\label{cond.ZS2}\label{ZSaction-restriction-assoc1} $h\cdot (xy)=(h\cdot x)(h|_x \cdot y)$,
\item\label{cond.ZS4}\label{ZSaction-restriction-assoc2} $(h_1h_2)|_x=h_1|_{h_2\cdot x} h_2|_x$.
\end{enumerate}
\begin{definition}[{\cite[Definition 1.1]{AA2005}}]\label{def:matchedpair}
    A pair of topological groupoids $(\cG,\cH)$ equipped with a continuous $\cH$-action map and continuous $\cG$-restriction map which satisfy \ref{cond.ZS1}--\ref{cond.ZS4} is called a {\em matched pair}.
\end{definition}

The following lemma is taken from \cite[Lemma 4]{BPRRW2017} and is especially useful in later calculations.

\begin{lemma}\label{lm.ZS} For any $h\in \cH$ and $x\in \cG$ with $s_\cH(h)=r_\cG(x)$, we have
\begin{enumerate}[leftmargin=1.5cm, label=\textup{(ZS\arabic*)}]\setcounter{enumi}{9}
\item\label{item:acting-on-unit-is-unit} $h\cdot s_\cH(h)=r_\cH(h)$,
\item\label{item:restriction-of-unit-is-unit} $r_\cG(x)|_x=s_\cG(x)$,
\item $(h\cdot x)^{-1}=h|_x \cdot x^{-1}$,
\item\label{item:inverse-of-restriction} $(h|_x)^{-1}=h^{-1}|_{h\cdot x}$.
\end{enumerate}
\end{lemma}

We can now define the  \ZS product groupoid as follows.

\begin{definition} Suppose $(\cG,\cH)$ is a matched pair. Define 
\[\cG\bowtie\cH := \cG \tensor*[_{s_{\cG}}]{\times}{_{r_{\cH}}}\cH=\{(x,h): x\in\cG, h\in\cH, s_\cG(x)=r_\cH(h)\},\]
with multiplicable pairs
\[(\cG\bowtie\cH)^{(2)} = \{((x,h),(y,g)): r_\cG(y)=s_\cH(h)\}.\]
Define multiplication
\[(x,h)(y,g)=(x(h\cdot y), h|_y g),\]
and inverse map
\begin{equation}\label{eq:inv-on-ZS-gpd}
(x,h)^{-1}=(h^{-1}\cdot x^{-1}, h^{-1}|_{x^{-1}}).\end{equation}
Then $\cG\bowtie\cH$ is a groupoid \cite{BPRRW2017}, called the {\em \ZS product} of $\cG$ and $\cH$.
\end{definition} 

We point out that the set of idempotents $(\cG\bowtie\cH)\z$ of $\cG\bowtie\cH$ can be identified with $\cG\z=\cH\z$ since one can prove
\begin{align}\label{eq:s-and-r-on-SZ}
    s_{\cG\bowtie\cH} (x,h)
    &=
    \bigl(
        s_{\cH}(h), s_{\cH}(h)
    \bigr)
    \text{ and }
    r_{\cG\bowtie\cH} (x,h)
    =
    (r_{\cG}(x),r_{\cG}(x)),
\end{align}
using Conditions~\ref{cond.ZS2}, \ref{cond.ZS3}, \ref{cond.ZS6}, and \ref{item:acting-on-unit-is-unit} for the first and \ref{cond.ZS1}, \ref{cond.ZS4}, \ref{cond.ZS8}, and \ref{item:restriction-of-unit-is-unit}  for the second equality.

Any self-similar group action on a groupoid gives rise to a \ZS product as follows.

\begin{example}\label{ex.ssgroupoid} 
     Let $\cG$ be an \etale groupoid, and let $\sigma\colon  H\to\operatorname{Aut}(\cG)$ be an action of a discrete group $H$ on $\cG$. Denote $\sigma(h)(x)$ by $h\ast x$. The action is called {\em self-similar} if there exists another map $H\times \cG \ni (h,x)\mapsto h\bullet x \in H$ that satisfies:
    \begin{enumerate}
        \item For all $v\in\cG\z $, $h\in H$, and $x\in\cG$, we have $h\bullet v=h$ and $e_H \bullet x =e_H$.
        \item For all  $(x,y)\in \cG^{(2)}$
        and $h\in H$, we have $h\bullet (xy)=(h\bullet x)\bullet y$.
        \item For all $(x,y)\in \cG^{(2)}$
        and $h\in H$, we have $h\ast (xy)=(h\ast x)((h\bullet x) \ast y)$.
        \item For all $g,h\in H$ and $x\in \cG$, we have $(gh)\bullet x=(g \bullet (h\ast x)) (h\bullet x)$.
    \end{enumerate}
    The {\em self-similar groupoid} is then defined as $\cG^H=\{(x,h): x\in \cG, h\in H\}$ with multiplication given by $(x,h)(y,k)=(x(h\ast y), (h\bullet y) k)$ whenever $h\ast r(y)=s(x)$, and with inverse $(x,h)^{-1}=(h^{-1}\ast x^{-1}, h^{-1} \bullet x^{-1})$. 

    Let $\cH=\cG\z \rtimes H=\{(u,h): u\in \cG\z , h\in H\}$ be the transformation groupoid of the $H$-action restricted to $\cG\z$; its multiplication is given by $(u,h)(v,k)=(u,hk)$ whenever $v=h^{-1}\ast u$, and its inverses are $(u,h)^{-1}=(h^{-1}\ast u, h^{-1})$. One can verify that $(\cG,\cH)$ is a matched pair of \etale groupoids where the $\cH$-action on $\cG$ is given by $(u,h)\cdot x := h\ast x$ and the $\cG$-restriction on $\cH$ is given by $(u,h)|_x=(h\ast s(x),h \bullet x)$. The resulting \ZS product groupoid $\cG\bowtie \cH$ coincides with the self-similar groupoid $\cG^H$.
    %
\end{example}

Note further that, when $\cG$ and $\cH$ are \etale groupoids, their \ZS product is again \etale. 

\begin{proposition}[see {\cite[Proposition 9]{BPRRW2017}}] When $\cG\bowtie\cH$ is endowed with the relative product topology on $\cG\times \cH$, it is \etale if and only if both $\cG$ and $\cH$ are \'{e}tale, and both the $\cH$-action map and the $\cG$-restriction maps are continuous. 
\end{proposition} 

\subsection{Fell Bundles}

Fell bundles over groups were first introduced and studied by Fell \cite{FellBundleBook}, under the notion of $C^*$-algebraic bundle. It is a powerful device in the study of graded $C^*$-algebras, and many well-known $C^*$-algebras are naturally graded. Instead of studying the graded $C^*$-algebra as a whole, Fell bundles focus on the fibres from the grading and provide a general framework to reassemble fibres back to various graded $C^*$-algebras. Here, we give a brief introduction to Fell Bundles over \etale groupoids. One may refer to \cite{ExelFellBundle} for Fell bundles over discrete groups and its connection with partial dynamical systems; and to \cite{Yamagami1990, Kumjian1998} for a more detailed discussion of Fell bundles over groupoids. 

\begin{definition}[see {\cite[Definition 2.1]{BE2009}}]\label{def.USCBdl}
Suppose $\cG$ is a locally compact 
Hausdorff
\etale groupoid, and $B$ is a topological space together with a continuous, open surjection $p\colon B\to\cG$. 
We call $\cB=(B,p)$ an {\em upper semi-continuous Banach bundle} if its fibres $\cB_{x}:=p^{-1}(x)$ have the structure of complex Banach spaces and if
\begin{enumerate}[leftmargin=2cm,label=(USC\arabic*)]
    \item\label{cond.B-USC} the map $B\to \mathbb{R}_{\geq 0}$, $b\mapsto\norm{b}$, is upper semi-continuous,
    \item\label{cond.B-plus} when
    \[
        B\tensor*[_{p}]{\times}{_{p}} B = \{(a,b)\in B\times B : p(a)=p(b)\}
    \]
    is equipped with the subspace topology, then the map  $B\tensor*[_{p}]{\times}{_{p}} B \to B$, $(a,b)\mapsto a+b$, is continuous, 
    \item\label{cond.B-times} for each $\lambda\in\mathbb{C},$ the map $B\to B,$ $b\mapsto \lambda b$, is continuous and
    \item\label{cond.B-nets} if $(b_i)_i$ is a net in $B$ such that $p(b_i)$ converges to $x\in \cG$ and $\norm{b_i}\to 0$, then $(b_i)_i$ converges to $0\in \cB_{x}$ in $B$.
\end{enumerate}
We note that \ref{cond.B-plus} and \ref{cond.B-times} are to be understood with the Banach space structure of $\cB_{p(b)}$ in mind. By a standard abuse of notation, we will often write $\cB=\{\cB_g\}_{g\in\cG}$ for the bundle, omitting explicit referencing of the topological space $B$ and the map $p$.
\end{definition}
\begin{definition}[see {\cite[Definition 2.8]{BE2009}}]
An upper semi-continuous Banach bundle $\cB=(B,p)$ is called {\em Fell bundle} (or {\em $C^*$-algebraic bundle}) if it comes with continuous maps
\[
    \cdot\colon
     \cB^{(2)}:=
     \left\{
        (a,b)\in B\times B  : (p(a),p(b))\in \cG^{(2)}
     \right\}
    \to
    B
    \text{ and }
    {}^{\ast}\colon B \to B
\]
such that:
\begin{enumerate}[label=\textup{(F\arabic*)}]
\item\label{cond.F1} For each $(x,y)\in \cG^{(2)}$, $\cB_{x}\cdot \cB_y\subset\cB_{xy}$, i.e.\ $p(b\cdot c)=p(b) p(c)$ for all $(b,c)\in  \cB^{(2)}$.
\item\label{cond.F2} The multiplication is bilinear.
\item\label{cond.F3} The multiplication is associative, whenever it is defined.
\item\label{cond.F4} If $(b,c)\in  \cB^{(2)}$, then $\|b\cdot c\|\leq\|b\|\|c\|$, where the norm is the Banach norm of the respective fibre. 
\item\label{cond.F5} For any $x\in \cG$, $\cB_{x}^*\subset \cB_{x^{-1}}$.
\item\label{cond.F6} The involution map $b\mapsto b^*$ is conjugate linear.
\item\label{cond.F7} If $(b,c)\in  \cB^{(2)}$, then $(b\cdot c)^*=c^*\cdot  b^*$. 
\item\label{cond.F8} For any $b\in B$, $b^{**}=b$.
\item\label{cond.F9} For any $b\in B$, $\|b^*\cdot  b\|=\|b\|^2=\|b^*\|^2$. 
\item\label{cond.F10} For any $b\in B$, $b^*\cdot  b\geq 0$ in $\cB_{s_{\cG}(p(b))}$. 
\end{enumerate}
We will often write $bc$ for $b\cdot c$, and $s_{\cB}$ resp.\ $r_{\cB}$ for $ s_{\cG}\circ p$ resp.\ $r_{\cG}\circ p$.
\end{definition}

We note that \condref{F9} makes sense because of
\condref{F5}. Moreover, whenever $x\in \cG\z$ is an idempotent, $\cB_{x}$ is in fact a $C^*$-algebra. Since $b^*b\in \cB_{s(g)}$  for $b\in \cB_g$ by \condref{F1} and \condref{F4}, we can understand the positivity in \condref{F10}
within the $C^*$-algebra $\cB_{s(g)}$. 
\begin{lemma}[cf.\ {\cite[Lemma 3.30]{BMZ2013}}]
    Suppose $\cG\z$ is discrete. A Fell bundle $(B,p)$ in the above sense is automatically continuous, i.e.\ the norm $b\mapsto \norm{b}$ is continuous on $B$.
\end{lemma}
\begin{proof} Let $\{b_i\}_{i\in I}$ be a net in $B$ converging to $b$. By the assumptions that involution and multiplication are continuous, $b_i^* b_i\to b^* b$. Since $p:B\to \cG$ is continuous, $p(b_i^* b_i)\to p(b^* b)$. Because $\cG\z$ is discrete, we may assume without loss of generality that $p(b_i^* b_i)=p(b^* b)$ for all $i$. Therefore, $\{b_i^* b_i\}$ is a net converging to $b^* b$ in the $C^*$-algebra $\cB_{p(b^* b)}=\cB_{s_\cB(b)}$. Consequently, $\|b_i\|^2=\|b_i^* b_i\|\to\|b^* b\|=\|b\|^2$. Therefore, $b\mapsto\|b\|$ is continuous on $B$.  
\end{proof}

\begin{example} Given an \etale groupoid $\cG$, one can define the groupoid Fell bundle $\cB(\cG):=\mathbb{C}\times\cG=\{(a,x): a\in\mathbb{C}, x\in \cG\}$ as follows. A fibre $\cB(\cG)_{x}=\mathbb{C}\times\{x\}$ naturally inherits its norm from $\mathbb{C}$. Multiplication is given by $(a,x)(b,y)=(ab,xy)$ whenever $(x,y)\in\cG^{(2)}$, and involution is given by $(a,x)^*=(\overline{a}, x^{-1})$. One can easily verify that $\cB(\cG)$ is a Fell bundle over $\cG$. 
\end{example} 

\begin{example} Let $G$ be a discrete group. Exel \cite[Definition 16.2]{ExelFellBundle} defined the notion of $C^*$-grading that is closely related to Fell bundles over $G$: For a $C^*$-algebra $A$, a {\em $C^*$-grading} is a collection of linearly independent subspaces $\{\cA_g\}_{g\in G}$ such that $\oplus_{g\in G} \cA_g$ is dense in $A$
, $\cA_g \cA_h\subset \cA_{gh}$, and $\cA_g^*\subset \cA_{g^{-1}}$. Given such a grading, $\cB=\{\cA_g\}_{g\in G}$
defines a Fell bundle over $G$, where the multiplication and involution are inherited from the underlying $C^*$-algebra.

One has to be cautious that by passing from the $C^*$-algebra $A$ 
to the Fell bundle $\{\cA_g\}$, one may lose much information of the original $C^*$-algebra $A$. As pointed out by Exel \cite[Remark 16.3]{ExelFellBundle}, there may be multiple ways of completing $\oplus_{g\in G} \cA_g$, some of which may not recover $A$. 
\end{example}

\begin{definition}[cf.\ {\cite[Definition II.13.8]{FellDoran:Vol1} and \cite[Definition VIII.3.3]{FellDoran:Vol2}}]\label{def:Fellbdlhom}
    Suppose $\cB=(B,p)$ and $\cB'=(B',p')$ are two  Fell bundles over groupoids $\cG$ and $\cG'$ respectively, and let $f\colon \cG\to\cG'$ be a continuous groupoid homomorphism. A continuous map $\phi\colon B\to B'$  between the total spaces is called a {\em homomorphism $\cB\to \cB'$ of Fell bundles covariant with $f$} if
    \begin{enumerate}[label=\textup{(H\arabic*)}]
        \item $\phi (\cB_x)\subset \cB'_{f(x)} $ for all $x\in \cG$ and each $\phi|_{\cB_x}$ is a linear map between the Banach spaces $\cB_x$ and $\cB_{f(x)}$,
        \item $\phi$ is multiplicative, i.e.\ if $(b,c)\in \cB^{(2)}$, then 
        $\phi(b c) = \phi(b)\phi(c)$, and
        \item $\phi$ is $*$-preserving, i.e.\ $\phi(b^*)=\phi(b)^*$ for all $b\in\cB$.
    \end{enumerate}
    If, moreover, $\norm{\phi(b)}=\norm{b}$ for all $b\in \cB,$ we call $\phi$ {\em isometric}. If $\phi$ is bijective \textup(and isometric\textup), $\cG=\cG'$, and $f$ is the identity, then we say that $\phi$ is an {\em \textup(isometric\textup) isomorphism of Fell bundles}.
\end{definition}

\section{\ZS Product of a Fell Bundle and a groupoid}

Suppose $(\cG,\cH)$ is a matched pair of \etale groupoids and $\cB=(B,p)$ is a Fell bundle over $\cG$.
The goal of this section is to define a \ZS product of 
$\cB$ by the groupoid $\cH$ and show that this \ZS product is a Fell bundle over the \ZS product $\cG\bowtie \cH$. The Fell bundle $\cB$ defines a `$\cB$-restriction map' on $\cH$ quite easily: it can simply inherit the $\cG$-restriction map on $\cH$ from its $\cG$-grading. However, we need to additionally assume a certain type of 
$\cH$-action on the Fell bundle.

\begin{definition}\label{df.beta}
Assume $(\cG,\cH)$ is a matched pair of \etale groupoids as described in Definition~\ref{def:matchedpair}, and assume $\cB=(B,p)$ is a Fell bundle over $\cG$. For $r_{\cB}:=r_{\cG}\circ p \colon B\to \cG\z=\cH\z$, let $\cH\tensor*[_{s_{\cH}}]{\times}{_{r_{\cB}}} B$ be defined as in Equation~\eqref{eq:fibred-product}, equipped with the subspace topology of $\cH\times B$.

A {\em $(\cG,\cH)$-compatible
$\cH$-action on $\cB$} is a continuous map
\[
    \cH\tensor*[_{s_{\cH}}]{\times}{_{r_{\cB}}} B
    \overset{\beta}{\to}
    B
\]
satisfying the following conditions:
\begin{enumerate}[label=\textup{(A\arabic*)}]
\item\label{cond.A1} For any $(h,x)\in \cH\tensor*[_{s_{\cH}}]{\times}{_{r_{\cG}}}\cG$, the map $\beta_h := \beta (h,\textvisiblespace)$ maps $\cB_{x}$ into $\cB_{h\cdot x}$ 
and is linear.
\item\label{cond.A2}For any  
$(g,h)\in\cH^{(2)}$, $\beta_{gh}=\beta_g\circ\beta_h$. 
\item\label{cond.A3} For any $u\in \cH\z$, $\beta_u$ is the identity map. 
\item\label{cond.A4} For any $(b,c)\in  \cB^{(2)}$ such that $(h,bc)\in \cH\tensor*[_{s_{\cH}}]{\times}{_{r_{\cB}}} B$, we have
\[
\beta_h(bc)=\beta_h(b) \beta_{h|_{p(b)}}(c).\]
\item\label{cond.A5} For any $b\in \cB_{x}$ with $r_\cG(x)=s_\cH(h)$, we have
\[\beta_h(b)^* = \beta_{h|_x}(b^*).\]
\end{enumerate}
\end{definition} 

Using both \ref{cond.F1}--\ref{cond.F10} and \ref{cond.ZS1}--\ref{item:inverse-of-restriction}, one can check that the above conditions make sense.

\begin{proposition}\label{prop.A6} For any $h\in\cH$, the restricted map $\beta_h\colon \cB_{s_\cH(h)} \to \cB_{r_\cH(h)}$ is an isometric $*$-isomorphism of $C^*$-algebras.
\end{proposition} 

\begin{proof} For any $h\in\cH$, take $a\in \cB_{s_\cH(h)}$. By \condref{A1}, $\beta_h(a)\in\cB_{h\cdot s_\cH(h)}=\cB_{r_\cH(h)}$. For any $a,b\in \cB_{s_\cH(h)}$, 
\[\beta_h(ab)=\beta_h(a) \beta_{h|_{s_\cH(h)}}(b)=\beta_h(a)\beta_h(b).\]
Moreover, by \condref{A5},
\[\beta_h(a)^*=\beta_{h|_{s_\cH(h)}}(a^*)=\beta_h(a^*).\]
Therefore, $\beta_h\colon\cB_{s_\cH(h)} \to \cB_{r_\cH(h)}$ is a $*$-homomorphism. By \condref{A2} and \condref{A3}, $\beta_{h^{-1}}\beta_h=\beta_{s_\cH(h)}$ is the identity on $\cB_{s_\cH(h)}$ and $\beta_h\beta_{h^{-1}}=\beta_{r_\cH(h)}$ is the identity on $\cB_{r_\cH(h)}$; thus, $\beta_h$ is a bijective $*$-isomorphism. Since $\cB_{s_\cH(h)}$ and $\cB_{r_\cH(h)}$ are $C^*$-algebras, it is automatic that $\beta_h$ is isometric on $\cB_{s_\cH(h)}$.
\end{proof} 

\begin{corollary}\label{cor.iso} For any $h\in \cH$ and $x\in\cG$ with $r_\cG(x)=s_\cH(h)$, $\beta_h$ is isometric from $\cB_{x}$ to $\cB_{h\cdot x}$. That is, for any $a\in \cB_{x}$, 
\[\|\beta_h(a)\|= \|a\|.\]
\end{corollary}

\begin{proof} By Proposition \ref{prop.A6}, $\beta_h$ is isometric on $\cB_{s_\cH(h)}$. Consider $aa^* \in \cB_{r_\cG(x)}=\cB_{s_\cH(h)}$, we have $\|\beta_h(aa^*)\|=\|aa^*\|=\|a\|^2$. On the other hand, by \condref{A4} and \condref{A5},
\[\|\beta_h(aa^*)\| = \| \beta_h(a) \beta_{h|_x}(a^*)\| = \| \beta_h(a) \beta_h(a)^*\| = \|\beta_h(a)\|^2.\]
Therefore, $\|\beta_h(a)\|=\|a\|$, as desired. 
\end{proof} 

From an $\cH$-action $\beta$ on the Fell bundle $(B,p)$, we now construct a Fell bundle $(C,q)$ over the locally compact Hausdorff \etale groupoid $\cG\bowtie \cH$ as follows.
\begin{enumerate}[label=(C\arabic*)]
    \item\label{cond.C1}  As a topological space, let
    \[
        C:=B\tensor*[_{s_{\cB}}]{\times}{_{r_{\cH}}} \cH
        =
        \left\{
            (b,h)\in B\times \cH:
            s_{\cB}(b)=r_{\cH} (h)
        \right\},
    \]
    equipped with the subspace topology. The constraint we put on elements of $C$ enables us to define the following map with values in the \ZS product:
    \[
        q\colon C \to \cG\bowtie \cH,
        \quad q(b,h) = (p(b),h).
    \]
    Let $\cC:= (C,q)$.
    \item\label{cond.C2}  We define a multiplication $\bullet\colon \cC^{(2)}\to C$ by
    \[
        (a,g)\bullet (b,h) := (a\beta_g(b), g|_{p(b)} h),
    \]
    where 
    \[\cC^{(2)}:=\{ ((a,g),(b,h))\in C\times C : (q(a,g),q(b,h))\in (\cG\bowtie\cH)^{(2)}\},\]
    as defined before for $\cB=(B,p)$.
    \item\label{cond.C3}  We define an involution ${}^* \colon C\to C$ by
    \[(b, h)^* = \left(\beta_{h^{-1}}(b^*), h^{-1}|_{p(b)^{-1}}\right).\]
\end{enumerate}

\begin{remark}\label{rm.pos}
The fibre $\cC_{(x,h)} := q\inv (x,h)$ of $\cC$ is canonically isomorphic to the fibre $\cB_{x}$ of $\cB$, making $\cC_{(x,h)}$ a complex Banach space. In particular, for $u\in \cG\z=\cH\z$, 
the fibre $\cC_u = \cC_{(u,u)}$ is a $C^*$-algebra. An element $(b, u)\in\cC_{u}$ is positive if and only if $b\in\cB_u$ is positive. 
\end{remark}

\begin{proposition}\label{prop.USCbundle}
    The pair $\cC=(C,q)$ is an upper semi-continuous Banach bundle over the \ZS product $\cG\bowtie \cH$. 
\end{proposition}

\begin{proof}
    The map $q$ is clearly a continuous open surjection, since $p$ is.
    
    As $\cC_{(x,h)} := q\inv (x,h)$ inherits its structure of a complex Banach space from $\cB_{x}$, we have, for $((a,h),(b,g))\in C\tensor*[_{q}]{\times}{_{q}} C$ and $\lambda\in\mathbb{C}$, that $h=g$ and $\lambda(a,h)+(b,g) = (\lambda a+b,h)$. It is now clear that addition on $ C\tensor*[_{q}]{\times}{_{q}} C$ and multiplication by a scalar $\lambda$ on $C$ are continuous since this is the case for $(B,p)$, proving \ref{cond.B-plus} and \ref{cond.B-times} of Definition~\ref{def.USCBdl}.
    
    Since
    \[
        \begin{tikzcd}[row sep = tiny]
            C \ar[r, "\mathrm{pr}_{1}"] \ar[rr, bend left=50, "\norm{\cdot}"{name=U}] & B \ar[to=U, phantom, "\circlearrowleft"]\ar[r, "\norm{\cdot}"] &\mathbb{R}_{\geq 0}
        \end{tikzcd}
    \]
    commutes, continuity of the coordinate projection $\mathrm{pr}_{1}$ and upper semi-continuity of $\cB$ imply upper semi-continuity of $\cC$, i.e.\ \ref{cond.B-USC} holds.
    
    Lastly, assume $(b_i, h_i)_i$ is a net in $C$ with $q(b_i,h_i)$ converging to $(x,h)\in \cG\bowtie\cH$ and $\norm{(b_i,h_i)}\to 0$, i.e.\ $p(b_i) \to x$, $h_i\to h$, and $\norm{b_i}\to o$. Since $(B,p)$ satisfies \ref{cond.B-nets}, it follows that $(b_i)_i$ converges to $0\in \cB_{x}$ in $B$, so that $(b_i, h_i)_i$ converges to $(0, h) = 0\in \cC_{(x,h)}$ in $C$. This shows that $\cC$ satisfies  \ref{cond.B-nets} and is all in all an upper semi-continuous Banach bundle over $\cG\bowtie\cH$.
\end{proof}

\pagebreak[3]\begin{lemma}\label{lem:mult-on-C}
    The multiplication $\bullet$ on $\cC$ is well-defined and continuous.
\end{lemma}

\begin{proof}
     Let $(a,g),(b,h)\in C$ such that $(q(a,g),q(b,h))\in (\cG\bowtie\cH)^{(2)}$, i.e.\ $s_{\cH}(g) =r_{\cB}(b)= r_{\cG}(p(b))$. This means, first of all, that $x:=p(b)\in \cG$ can act on $g$, yielding $g\vert_{x }\in \cH$. Moreover, it means that $(g,b)$  is in $\cH\tensor*[_{s_{\cH}}]{\times}{_{r_{\cB}}} B $, the domain of $\beta$, so that $\beta_{g}(b)$ is defined. As $(a,g)\in C$, we have $s_{\cB}(a)=r_{\cH}(g)$, so that
    \begin{align*}
        s_{\cB}(a)
        &\overset{\condref{ZS5}}{=}
        r_{\cG}(g\cdot x )
        \overset{\condref{A1}}{=}
        r_{\cG}(p(\beta_{g}(b))),
    \end{align*}
    i.e.\ $(a, \beta_{g}(b))\in \cB^{(2)}$, so that their product $a\beta_{g}(b)$ is defined. Similarly, since $(b,h)\in C$, we have
    \[
        r_{\cH}(h) = s_{\cG} (x ) \overset{\condref{ZS5}}{=} s_{\cH} (g\vert_{x }),
    \]
    so that $(g\vert_{x },h)\in\cG^{(2)},$ i.e.\ their product $g\vert_{x }h$ is defined in $\cG$. All in all, it makes sense to define
    \[(a,g)\bullet (b,h) := (a\beta_g(b), g|_{x } h),\]
    and it remains to show that it is an element of $C$. To this end, we note that  $p(a\beta_g(b))=p(a)p(\beta_g(b))$ by \condref{F1} and $p(\beta_g (b))= g\cdot p(b)=g\cdot x$ by \condref{A1}, so that
    \begin{align*}
        s_{\cB} ( a\beta_g(b))
        =
        s_{\cG} (p(\beta_g(b)))
        =
        s_{\cG} (g\cdot x)
        \overset{\condref{ZS7}}{=}
        r_{\cH} (g|_{x } h).
    \end{align*}
    Thus, 
    \[
        (a,g)\bullet (b,h)
        \in
        B\tensor*[_{s_{\cB}}]{\times}{_{r_{\cH}}} \cH
        =
        C.
    \]
    It is now obvious that $\bullet$ is continuous, as the $\cH$-action $\beta$ on $B$, the multiplication on $B$, the bundle map $p$, the $\cG$-restriction on $\cH$, and the multiplication on $\cH$ are all continuous maps. 
\end{proof}

\begin{lemma}\label{lem:invol-on-C}
    The involution ${}^*$ on $\cC$ is well-defined and continuous.
\end{lemma}

\begin{proof}
    Let $(b,h)\in C$ with $x:= p(b)$, i.e.\ $s_{\cB}(b)=r_{\cH}(h)$. As $p(b^*)= x\inv$ by \condref{F5}, we have
    \[
        s_{\cH} (h\inv) = r_{\cH} (h) 
        =
        s_{\cB}(b)
        =
        s_{\cG} (x)
        =
        r_{\cG} (x\inv)
        =
        r_{\cB} (b^*).
    \]
    This shows both that $(h\inv,x\inv)\in \cH \tensor*[_{s_{\cH}}]{\times}{_{r_{\cG}}}\cG$, so that $h\inv\vert_{x\inv}$ is defined, and that $(h\inv,b^*)$ is in $\cH\tensor*[_{s_{\cH}}]{\times}{_{r_{\cB}}} B $, the domain of $\beta$, so that $\beta_{h\inv}(b^*)$ is defined. All in all, it makes sense to define
    \[(b, h)^* := \left(\beta_{h^{-1}}(b^*), h^{-1}|_{x^{-1}}\right),\]
    and it remains to show that it is an element of $C$. To this end, we compute
    \begin{align*}
        s_{\cB}(\beta_{h^{-1}}(b^*))
        \overset{\condref{A1}}{=}
        s_{\cG} (h\inv \cdot p (b^*))
        =
        s_{\cG} (h^{-1}\cdot x^{-1})
        \overset{\condref{ZS7}}{=}
        r_{\cH} (h^{-1}|_{x^{-1}}).
    \end{align*}
    Thus, 
    \[
        (b,h)^*
        \in
        B\tensor*[_{s_{\cB}}]{\times}{_{r_{\cH}}} \cH
        =
        C.
    \]
    It is now obvious that ${}^*$ is continuous, as the $\cH$-action $\beta$ on $B$, involution on $B$, the bundle map $p$, the $\cG$-restriction on $\cH$, and inversion on $\cH$ are all continuous maps.
\end{proof}

\begin{theorem}\label{thm.bundle}
    The bundle $\cC$
    defined in \ref{cond.C1}, together with the multiplication defined in \ref{cond.C2} and the involution defined in \ref{cond.C3}, is a Fell bundle over the \ZS product $\cG\bowtie \cH$. We will denote $\cC$ by $\cB\bowtie_\beta \cH$, and we call it the \ZS product of the Fell bundle $\cB$ by $\cH$.
\end{theorem}

\begin{proof}
We have already seen in Proposition~\ref{prop.USCbundle} that $\cC$ is an upper semi-continuous Banach bundle over $\cG\bowtie\cH$, and in Lemmas~\ref{lem:mult-on-C} and~\ref{lem:invol-on-C} that $\bullet$ resp.\ ${}^*$ are well-defined continuous $C$-valued maps.
It remains to show that $\cC$
satisfies conditions \hyperref[cond.F1]{(F1)} through \hyperref[cond.F10]{(F10)}.

For \hyperref[cond.F1]{(F1)}: Take $ ( a , h ) \in \cC_{(x,h)}$ and $ ( b , g ) \in \cC_{(y,g)}$ such that $r_\cG(y)=s_\cH(h)$, so that
\[ ( a , h )  ( b , g ) =(a\beta_h(b), h|_y g).\]
By \ref{cond.A1}, $\beta_h(b)\in \cB_{h\cdot y}$. By \ref{cond.ZS5}, $r_\cG(h\cdot y)=r_\cH(h)=s_\cG(x)$, and thus $(x,h\cdot y)\in \cG^{(2)}$. Therefore, by \condref{F1} for $\cB$, $a\beta_h(b)\in \cB_{x(h\cdot y)}$. Hence, the product $(a\beta_h(b) , h|_y g)$ is an element of $\cC_{(x(h\cdot y), h|_y g)}=\cC_{(x,h)(y,g)}.$

For \condref{F2}: The multiplication on $B$ is bilinear and $\beta_h$ is linear, so it is clear that the multiplication on $C$ is bilinear. 

For \condref{F3}: Take $ ( a , h ) \in \cC_{(x,h)}$, $ ( b , g ) \in \cC_{(y,g)}$, and $ ( c , k ) \in \cC_{(z,k)}$, such that $r_\cG(y)=s_\cH(h)$ and $r_\cG(z)=s_\cH(g)$. By definition,
\[\bigl( ( a , h )  ( b , g ) \bigr)\, ( c , k )=\bigl(a\beta_h(b)\beta_{h|_y g}(c) , (h|_y g)|_z k\bigr).\]
On the other hand, 
\[ ( a , h )\, \bigl( ( b , g )  ( c , k ) \bigr)=\bigl(a\beta_h(b\beta_g(c)) , h|_{y(g\cdot z)} g|_z k\bigr).\]
By \condref{ZS4} and \condref{ZS3}, 
\[(h|_y g)|_z k = (h|_y)|_{g\cdot z} g|_z k = h|_{y(g\cdot z)} g|_z k.\]
By \condref{A4}, 
\[a\beta_h(b\beta_g(c))=a\beta_h(b)\beta_{h|_y}(\beta_g(c))=a\beta_h(b)\beta_{h|_y g}(c).\]
Therefore, the multiplication is associative. 

For \condref{F4}: For $ ( a , h ) \in \cC_{(x,h)}$ and $ ( b , g ) \in \cC_{(y,g)}$ such that $r_\cG(y)=s_\cH(h)$, we have 
\begin{align*}
\| ( a , h )  ( b , g ) \| &= \|(a\beta_h(b) , h|_y g)\| \\
&=\|a\beta_h(b)\| \\
&\leq \|a\|\|\beta_h(b)\| \\
&=\|a\|\|b\| = \| ( a , h ) \|\|(b,g)\|.
\end{align*}
Here, we applied Corollary~\ref{cor.iso} which stated that $\|\beta_h(b)\|=\|b\|$.

For \condref{F5}: For $ ( a , h ) \in \cC_{(x,h)}$, we have
\[ ( a , h ) ^* = \bigl(\beta_{h^{-1}}(a^*), h^{-1}|_{x^{-1}}\bigr) \in \cC_{(h^{-1}\cdot x^{-1}, h^{-1}|_{x^{-1}})} = \cC_{(x,h)^{-1}}.\]

For \condref{F6}: Since $\beta_{h^{-1}}$ is linear, we have  
\[(a+\lambda b,  h)^*= \bigl(\beta_{h^{-1}}(a^*+\overline{\lambda}b^*), h^{-1}|_{x^{-1}}\bigr) =  ( a , h ) ^* + \overline{\lambda}(b, h)^*.\]

For \condref{F7}: Take $ ( a , h ) \in \cC_{(x,h)}$ and $ ( b , g ) \in \cC_{(y,g)}$ such that $r_\cG(y)=s_\cH(h)$. One can compute:
\[( ( a , h )  ( b , g ) )^*=\bigl( \beta_{(h|_y g)^{-1}}(\beta_h(b)^* a^*) , (h|_y g)^{-1}|_{(x(h\cdot y))^{-1}}\bigr),\]
and
\[ ( b , g ) ^*  ( a , h ) ^*=
\Bigl(\beta_{g^{-1}}(b^*) \beta_{g^{-1}|_{y^{-1}}}(\beta_h^{-1}(a^*)) , ({g^{-1}|_{y^{-1}}})|_{h^{-1}\cdot  x^{-1}} h^{-1}|_{x^{-1}}\Bigr).\]
By \condref{A5}, $\beta_h(b)^*=\beta_{h|_y}(b^*)$. By Lemma \ref{lm.ZS}, $(h|_y)^{-1}=h^{-1}|_{h\cdot y}$ and $(h\cdot y)^{-1}=h|_y y^{-1}$. Therefore,
\begin{align*}
\beta_{(h|_y g)^{-1}}(\beta_h(b)^* a^*) &= \beta_{(h|_y g)^{-1}}(\beta_{h|_y}(b^*) a^*) \\
&= \beta_{(h|_y g)^{-1}}(\beta_{h|_y}(b^*)) \beta_{(h|_y g)^{-1}|_{h|_y\cdot y^{-1}}}(a^*) \\
&=\beta_{g^{-1}}(b^*) \beta_{(g^{-1}h^{-1}|_{h\cdot y})|_{(h\cdot y)^{-1}}}(a^*) \\
&=\beta_{g^{-1}}(b^*) \beta_{g^{-1}|_{(h^{-1}|_{h\cdot y})\cdot (h\cdot y)^{-1}} (h^{-1}|_{h\cdot y})|_{(h\cdot y)^{-1}}}(a^*) \\
&=\beta_{g^{-1}}(b^*) \beta_{g^{-1}|_{(h|_y)^{-1}\cdot h|_y \cdot y^{-1}} h^{-1}|_{r_\cH(h)}}(a^*) \\
&= \beta_{g^{-1}}(b^*) \beta_{g^{-1}|_{y^{-1}}}(\beta_{h^{-1}}(a^*))
.
\end{align*}
Moreover, 
\begin{align*}
(h|_y g)^{-1}|_{(x(h\cdot y))^{-1}} &= g^{-1}|_{(h|_y)^{-1}\cdot (x(h\cdot y))^{-1}} (h|_y)^{-1}|_{(x(h\cdot y))^{-1}} \\
&= g^{-1}|_{(h|_y)^{-1}\cdot (h|_y\cdot y^{-1}) x^{-1}} (h^{-1}|_{h\cdot y})|_{(h\cdot y)^{-1}x^{-1}} \\
&= g^{-1}|_{y^{-1} (h|_{h\cdot y})|_{(h\cdot y)^{-1}}\cdot x^{-1}} h^{-1}|_{x^{-1}} \\ 
&= g^{-1}|_{y^{-1} (h^{-1}|_{r_\cH(h)}\cdot x^{-1})} h^{-1}|_{x^{-1}} \\ 
&= (g^{-1}|_{y^{-1}})|_{h^{-1}\cdot x^{-1}}  h^{-1}|_{x^{-1}}. \\ 
\end{align*}
Therefore,
\[( ( a , h )  ( b , g ) )^*= ( b , g ) ^*  ( a , h ) ^*.\]

For \condref{F8}: Take any $ ( a , h ) \in \cC_{x,h}$, one can compute that
\[ ( a , h ) ^{**} = \Bigl(\beta_{(h^{-1}|_{x^{-1}})^{-1}} ( \beta_{h^{-1}} (a^*)^* ) , (h^{-1}|_{x^{-1}})^{-1}|_{(h^{-1}\cdot x^{-1})^{-1}}\Bigr).\]
By \condref{A5}, 
\[\beta_{h^{-1}} (a^*)^* = \beta_{h^{-1}|_{x^{-1}}}(a).\]
Therefore, 
\begin{align*}
\beta_{(h^{-1}|_{x^{-1}})^{-1}} ( \beta_{h^{-1}} (a^*)^* ) &= \beta_{(h^{-1}|_{x^{-1}})^{-1}} ( \beta_{h^{-1}|_{x^{-1}}}(a) ) \\
&= \beta_{r_\cG(x)}(a)=a.
\end{align*}
Moreover, 
\[(h^{-1}|_{x^{-1}})^{-1}|_{(h^{-1}\cdot x^{-1})^{-1}} = (h|_{h^{-1} \cdot x^{-1}})|_{(h^{-1}\cdot x^{-1})^{-1}}=h.\]
Hence, $( a , h ) ^{**}= ( a , h )$.

For \condref{F9}: For $ ( a , h ) \in \cC_{x,h}$, we have
\[ ( a , h ) ^*  ( a , h )  = \bigl(\beta_{h^{-1}}(a^*) \beta_{h^{-1}|_{x^{-1}}}(a) , s_\cH(h)\bigr)=\bigl(\beta_{h^{-1}}(a^*a), s_\cH(h)\bigr).\] 
Note that $a^*a\in \cB_{s_\cG(x)}=\cB_{r_\cH(h)}=\cB_{s_\cH(h^{-1})}$. By Proposition \ref{prop.A6}, $\beta_{h^{-1}}$ is isometric on  $\cB_{s_\cH(h^{-1})}$. Therefore, 
\[\| ( a , h ) ^*  ( a , h ) \|=\|\beta_{h^{-1}}(a^*a)\|=\|a^*a\|=\| ( a , h ) \|^2.\]
Moreover, 
\[\|a^*a\|=\|aa^*\|=\| ( a , h ) ^{**}  ( a , h ) ^*\|=\| ( a , h ) ^*\|^2.\]

Finally, for \condref{F10}: We have shown that 
\[ ( a , h ) ^*  ( a , h )  =\bigl( \beta_{h^{-1}}(a^*a),s_\cH(h)\bigr).\] 
Since $\beta_{h^{-1}}$ is a $*$-automorphism on $\cB_{s_\cH(h^{-1})}$, we have $\beta_{h^{-1}}(a^*a)\geq 0$ and thus by Remark~\ref{rm.pos}, \[ ( a , h ) ^*  ( a , h ) \geq 0. \qedhere\]
\end{proof} 

\begin{example}\label{ex:if-H-is-trivial}
    If $\cH=\cG\z$ is trivial, so that $\cG\bowtie\cG\z \cong \cG$ via $(x,s_{\cG}(x))\mapsto x$, then  the trivial $\cG\z$-action on some Fell bundle $\cB$ over $\cG$, defined by $\beta(r_{\cB}(b),b)=b$ for all $b\in\cB$, is clearly $(\cG,\cG\z)$-compatible, and $\cB\bowtie\cG\z \cong \cB$ via $(b,s_{\cG}(b)) \mapsto b$.
\end{example} 

\begin{example}\label{ex.ZSbundle} On the other hand, if $(\cG,\cH)$ is a matched pair of \etale groupoids, then the groupoid Fell bundle
$\cB(\cG):=(\mathbb{C}\times \cG, \mathrm{pr}_{\cG})$ carries a canonical $\cH$-action $\beta$: for $h\in \cH$ and $x\in\cG$ with $r_\cG(x)=s_\cH(h)$, define $\beta(h,( z ,x))=( z ,h\cdot x)$. This is clearly linear and continuous, and one can easily verify that $\beta$ satisfies conditions \condref{A1} through \condref{A5}. 

By Theorem \ref{thm.bundle}, the bundle $\cC$ with fibres
\[\cC_{(x,h)} = \{(( z ,x), h): ( z ,x)\in\cB_{x}, r_\cH(h)=s_\cG(x)\}\]
for $(x,h)\in \cG\bowtie\cH$, with multiplication
\[(( z_{1} ,x), h)(( z_{2}  ,y), g) = ((z_{1}z_{2},x h\cdot y), h|_y g),\]
and with involution
\[(( z ,x), h)^* = ((\overline{ z }, h^{-1}\cdot x^{-1}), h^{-1}|_{x^{-1}}),\]
is a Fell bundle over $\cG\bowtie\cH$. Observe that the map $(( z ,x),h)\mapsto ( z , (x,h))$ defines an isometric isomorphism between the Fell bundle $\cC=\cB(\cG)\bowtie_\beta \cH$ and the groupoid Fell bundle $\cB(\cG\bowtie\cH)$. 
\end{example} 

\begin{example}\label{ex.KMWQ}
    Let $\cG$ be a locally compact Hausdorff groupoid and $H$ a locally compact group. Suppose there is a continuous left action of $H$ on $\cG$, denoted by $(h,x)\mapsto h\ast x$.  Recall the {\em semi-direct product groupoid} $\cG\rtimes H$ considered by  Kaliszewski, Muhly, Quigg, and Williams in \cite{KMQW2010},
\[\cG\rtimes H=\{(x,h): x\in \cG, h\in H\},\]
with multiplication $(x,h)(y,k)=(x(h\ast y), hk)$ whenever $s_\cG(x)=h\ast r_\cG(y)$ and with inverse $(x,h)^{-1}=(h^{-1}\ast x^{-1}, h^{-1})$ (the reader should compare this construction to that of the self-similar groupoid in Example~\ref{ex.ssgroupoid}). The notion of a semi-direct product of a Fell bundle $\cB=(B,p)$ over $\cG$ by a group $H$ in \cite{KMQW2010} is related to our construction as follows.   
 
In \cite{KMQW2010}, an {\em $H$-action} on $\cB$ is defined as a homomorphism $\alpha\colon H\to \operatorname{Aut}(\cB)$ such that $H\times B\ni (h,b)\mapsto \alpha_h(b) \in B$ is continuous and $p(\alpha_h(b))=h\ast p(b)$ for all $h\in H$ and $b\in B$. They form the {\em semi-direct product Fell bundle}
\[\cB\rtimes_\alpha H=\{(b,h): b\in B, h\in H\}\]
with multiplication $(b,h)(c,k)=(b\alpha_h(c), hk)$ whenever $s_\cB(b)=r_\cG(h\ast p(c))$ and with involution $(b,h)^*=(\alpha_{h^{-1}}(b)^*, h^{-1})$. They proved that $\cB\rtimes_\alpha H$ is a Fell bundle over the semi-direct product groupoid $\cG\rtimes H$ \cite[Proposition 6.2]{KMQW2010}.

In the special case when $\cG$ is \etale and $H$ is discrete, one can verify that this construction coincides with ours in the following way.
 Let $\cH=\cG\z \rtimes H=\{(u,h): u\in \cG\z , h\in H\}$ be the transformation groupoid of the $H$-action restricted to $\cG\z$ with multiplication $(u,h)(v,k)=(u,hk)$ whenever $v=h^{-1}\ast u$ and with inverse $(u,h)^{-1}=(h^{-1}\ast u, h^{-1})$. Then $(\cG,\cH)$ is a matched pair of \etale groupoids where the $\cH$-action on $\cG$ is given by $(u,h)\cdot x:=h\ast x$ and the $\cG$-restriction on $\cH$ is given by $(u,h)|_x:=(h\ast s(x),h)$.
Moreover, the \ZS product groupoid $\cG\bowtie \cH$ coincides with the semi-direct product groupoid $\cG\rtimes H$. The $H$-action $\alpha$ can be lifted to a $(\cG,\cH)$-compatible $\cH$-action $\beta$ by setting $\beta_{(u,h)}(b)=\alpha_h(b)$. Furthermore, the \ZS product Fell bundle $\cB\bowtie_\beta \cH$ is isomorphic to the semi-direct product Fell bundle $\cB\rtimes_\alpha H$. 
\end{example}

We note that, in general, the Fell bundle $\cB$ is preserved isometrically inside $\cB\bowtie_\beta \cH$, but that we must make an additional assumption in order for it to also contain a copy of $\cB(\cH)$.

\begin{proposition}\label{prop.embed} 
    Define $\Phi\colon B\to B\tensor*[_{s_{\cB}}]{\times}{_{r_{\cH}}} \cH$ by $\Phi(b)=(b, s_\cB(b))$. Then $\Phi$ is an isometric homomorphism $\cB\to\cB\bowtie_{\beta}\cH$ of Fell bundles, covariant with the embedding $\cG\to \cG\bowtie \cH$. 
    Furthermore, $\Phi(\cB_{x})=(\cB\bowtie_{\beta}\cH)_{(x, s_\cG(x))}$.
    
    If we assume that $\cB_u$ is unital for all $u\in \cG\z$, then we may further define $\Psi\colon \mathbb{C}\times \cH \to  B \tensor*[_{s_{\cB}}]{\times}{_{r_{\cH}}} \cH$ by $\Psi(z,h)=(z 1_{u}, h)$ where $u=r_{\cH}(h)$. Then $\Psi$ is an isometric homomorphism $\cB(\cH) \to \cB\bowtie_\beta \cH$ of Fell bundles, covariant with the embedding $\cH\to \cG\bowtie \cH$.
\end{proposition} 

\begin{proof} 
    Let $\cC:=\cB\bowtie_{\beta}\cH$.
    Since $r_\cH(s_\cB(b))=s_\cB(b)$, 
    $\Phi$ indeed takes values in $B\tensor*[_{s_{\cB}}]{\times}{_{r_{\cH}}} \cH$. Clearly $\Phi$ is continuous and restricts to an isometric isomorphism between the Banach spaces $\cB_{x}$ and $\cC_{(x,s_\cG(x))}$, since the norm on the latter is given by
    \[
    \|(b, s_\cB(b))\|_{\cC_{q(\Phi(b))}}=\|b\|_{\cB_{p(b)}}.\]
    Using \condref{ZS8}, \condref{ZS9}, \ref{item:restriction-of-unit-is-unit}, and the fact that $\beta_u$ is the identity map, one can easily verify that $\Phi$ is $*$-preserving and multiplicative.

The map $\Psi$ is continuous and by construction covariant with the embedding. Since the fibre $\cC_{(u,h)}$
inherits its Banach space structure from $\cB_u$, $\Psi$ restricts to a linear and isometric map between the fibres. To see that $\Psi$ is multiplicative, let $\bigl((z,h),(w,k)
\bigr)\in\cB(\cH)^{(2)}$, i.e.\ $v:=s_{\cH}(h)=r_{\cH}(k)$. If $u:=r_{\cH}(h)$, then
\[
\bigl(    \Psi (z,h), \Psi (w,k)
\bigr) =\bigl(   (z 1_{u},h), (w1_{v},k)
\bigr)
\]
is in $\cC^{(2)}$, and according to Condition~\ref{cond.C2}, their product is given by 
\[
        \Psi (z,h) \bullet \Psi (w,k)
 = (z 1_{u}\, \beta_h(w1_{v}), h|_{p(w1_{v})} k).
\]
Since $\beta_h \colon \cB_{v} \to \cB_{u}$ is automorphic on $\cB_{s_\cH(h)}$ and thus unital
and since $h|_{p(w1_{v})}=h|_{v}=h$ by \ref{ZSG-restriction-unit}, we see that
\[
      \Psi (z,h)\bullet \Psi (w,k)
  = ((zw) 1_{u}, hk) = \Psi \bigl((z,h)(w,k)
\bigr).
\]
For the same reasons, we also have $\Psi(z,h)^*=\Psi(\overline{z}, h^{-1})=\Psi((z,h)^*)$. 
\end{proof}

We would like to analyze under which conditions a Fell bundle $\cC=(C,q)$ over the \ZS product groupoid $\cK=\cG\bowtie \cH$ can be decomposed as a \ZS product $\cB\bowtie_\beta \cH$ of some Fell bundle $\cB$ over $\cG$ with $(\cG,\cH)$-compatible $\cH$-action $\beta$. First, the map $\iota\colon\cG\to\cK=\cG\bowtie\cH$ given by $\iota(x)=(x,s_\cG(x))$ defines a continuous groupoid homomorphism, and thus the pullback bundle $\cB=\iota^*(\cC)=\{(x,c)\in\cG\times \cC\colon\iota(x)=q(c)\}$ with the map $p\colon\cB\to\cG$, $p(x,c)=x$ is a Fell bundle over $\cG$ (\cite[Remark 2.6]{Kumjian1998}). Here, each fibre $\cB_x$ can be identified as $\cC_{\iota(x)}=\cC_{(x,s_\cG(x))}$.

\begin{definition} Let $\cC=(C,q)$ be a Fell bundle over $\cK$ and $\cH$ a wide subgroupoid of $\cK$ , let $\j\colon \cH\to\cK$ denote the inclusion. We call a continuous section $u\colon \cH \to \j^*(\cC)$ of the pullback bundle $\j^*(\cC)$ an {\em $\cH$-unitary family}  
in $\cC$ if 
\begin{enumerate}[label=\textup{(U\arabic*)}]
    \item For any $(h,k)\in\cH^{(2)}$, $u_h u_k=u_{hk}$ and $u_h^* = u_{h^{-1}}$.
    \item For each $v\in\cH\z$, $u_v=1_v$ is the identity on $\cC_v$. 
\end{enumerate}
\end{definition}

\begin{proposition}\label{prop:decomposed-beta} Let $\cC=(C,q)$ be a Fell bundle over $\cK=\cG\bowtie \cH$ and suppose that there exists an $\cH$-unitary family $u$ \textup(where we identify $h\in\cH$ with $(r_\cH(h),h)\in\cK$\textup). Let $\cB=(B,p)=\iota^*(\cC)$ be the pullback bundle along the inclusion $\iota\colon\cG\to\cG\bowtie\cH$.
For any $(h,x)\in \cH\tensor*[_{s_{\cH}}]{\times}{_{r_{\cB}}}B$, define $\beta_h\colon\cB_x\to\cB_{h\cdot x}$ by 
\[\beta_h(a) = u_h a u_{h|_x}^*.\]
Then $\beta$ is a $(\cG,\cH)$-compatible $\cH$-action on $\cB$.
\end{proposition} 

\begin{proof} First of all, let us verify that $\beta_h$ is well-defined: by assumption, $u_h$ is an element of $\cC_{(r_\cH(h),h)}$, so that $u_{h|_x}^*=u_{(h|_x)^{-1}}\in \cC_{(s_\cH(h|_x),(h|_x)^{-1})}$. Since $s_\cH(h)=r_\cG(x)$ and $s_\cH(h|_x)=s_\cG(x)$ by \ref{cond.ZS5}, we have for $a\in \cB_x=\cC_{(x,s_\cG(x))}$ that $ u_h a u_{h|_x}^*$ is well-defined and is an element of the fibre of $\cC$ over
\[
    (r_\cH(h),h)(x,s_\cG(x))(s_\cG(x),(h|_x)^{-1})=(h\cdot x,h|_x (h|_x)^{-1})=(h\cdot x, s_\cG(h\cdot x)).
\]
Let us now verify that $\beta$ satisfies Definition \ref{df.beta}.
For \ref{cond.A1}, since $h\mapsto u_h$ is continuous and multiplication is continuous, we have that $\beta\colon (h,a)\mapsto u_h a u_{h|_x}^*$ is continuous. 

For \ref{cond.A2}, pick any $(g,h)\in\cH^{(2)}$, any $x\in\cG$ with $s_\cH(h)=r_\cG(x)$, and any element $b\in \cB_x$. Then
\[\beta_{gh}(b)=u_{gh} b u_{(gh)|_x}^*=u_g u_h b u_{h|_x}^* u_{g|_{h\cdot x}}^*=u_g \beta_h(b) u_{g|_{h\cdot x}}^* = \beta_g(\beta_h(b)).\]

For \ref{cond.A3}, for each $v\in \cH\z$, we have $u_v=1_v$, and so $\beta_v$ is clearly the identity map.

For \ref{cond.A4}, for any $(b,c)\in\cB^{(2)}$ such that $(h,bc)\in \cH\tensor*[_{s_{\cH}}]{\times}{_{r_{\cB}}} B$, we have
\[\beta_h(bc)=u_h (bc) u_{h|_{p(b)p(c)}}^* = u_h b u_{h|_{p(b)}}^* u_{h|_{p(b)}} c u_{(h|_{p(b)})|_{p(c)}}^*=\beta_h(b)\beta_{h|_{p(b)}}(c).\]

For \ref{cond.A5}, for any $b\in \cB_x$ with $r_\cG(x)=s_\cH(h)$, we have
\[\beta_h(b)^* = (u_h b u_{h|_x}^*)^*=u_{h|_x} b^* u_h^*.\]
Since $b^*\in \cB_{x^{-1}}$ and $(h|_x)|_{x^{-1}}=h|_{r_\cG(x)}=h$, we conclude that
\[\beta_h(b)^* =u_{h|_x} b^* u_h^*=\beta_{h|_x}(b^*). \qedhere\]
\end{proof}

We now prove that the \ZS product bundle $\cB\bowtie_\beta \cH$ is isometrically $*$-isomorphic to the original bundle $\cC$ over $\cG\bowtie\cH$. 

\begin{theorem}\label{thm.internalZS} Let $\cC=(C,q)$ be a Fell bundle over $\cG\bowtie \cH$ and suppose that there exists an $\cH$-unitary family $u$. With $\cB=(B,p)=\iota^* (\cC)$ the pull back bundle along $\iota\colon \cG\to \cG\bowtie \cH$, define $\Theta\colon B\tensor*[_{s_{\cB}}]{\times}{_{r_\cH}} \cH \to C$ by 
$\Theta(a,h)=au_h$. Then $\Theta$ is an isometric isomorphism $\cB\bowtie_\beta\cH\to\cC$ of Fell bundles over $\cG\bowtie\cH$, where $\beta$ is as defined in Proposition~\ref{prop:decomposed-beta}. 
\end{theorem} 

\begin{proof} 
    First of all, $\Theta$ clearly preserves fibres and is fibrewise linear. It is further continuous, by definition of the topology on the pullback bundle $\cB$, by the assumption that $h\mapsto u_h$ is continuous, and since multiplication on $\cC$ is continuous.
To see that $\Theta$ is multiplicative, take any $(a,h),(b,g)\in(\cB\bowtie_\beta \cH)^{(2)}$. By definition, $(a,h)(b,g)=(a\beta_h(b),h|_{p(b)} g)$. On the other hand,
\begin{align*}
    \Theta((a,h))\Theta((b,g)) &=au_h bu_g  \\
    &= au_h b u_{h|_{p(b)}}^* u_{h|_{p(b)}} u_g \\
    &= a\beta_h(b) u_{h|_{p(b)}g}=\Theta((a,h)(b,g)). 
\end{align*}
To see that $\Theta$ is *-preserving, take $(a,h)\in\cB\bowtie_\beta \cH$ with $a\in\cB_x$,  so that $(a,h)^*=(\beta_h^{-1}(a^*),h^{-1}|_{x^{-1}})$. We have,
\[\Theta((a,h))^*=u_{h^{-1}} a^*=u_{h^{-1}} a^* u_{h^{-1}|_{x^{-1}}}^* u_{h^{-1}|_{x^{-1}}}=\beta_{h^{-1}}(a^*)u_{h^{-1}|_{x^{-1}}}=\Theta((a,h)^*). \]
To see that $\Theta$ is isometric, 
\[\|\Theta((a,h))\|=\|au_h\|=\|au_h u_h^* a^*\|^{1/2}=\|aa^*\|^{1/2}=\|a\|=\|(a,h)\|.\]
In particular, $\Theta$ is automatically injective. Finally, for any $c\in \cC_{(x,h)}$, 
\[c u_h^*\in\cC_{(x,h)(s_\cH(h),h^{-1})}=\cC_{(x,hh^{-1})}=\cC_{(x,s_\cG(x))}\cong \cB_x.\]
Therefore,
\[c=c u_h^* u_h = \Theta((cu_h^*, h)).\]
As a result, $\Theta$ is surjective. 
\end{proof}

As an immediate corollary, we obtain the following version of the internal \ZS product of Fell bundles.

\begin{corollary} Let $\cC$ be a Fell bundle over $\cK$, and let $\cG,\cH$ be subgroupoids of $\cK$. Let $\cB$ be the pullback bundle along the inclusion $\cG\to\cK$, and suppose that there exists an $\cH$-unitary family $u$. Suppose for each $c\in\cC$, there exist unique $b\in\cB$ and $h\in\cH$ with $c=bu_h$. Then
$\cK=\cG\bowtie\cH$ and $\cC$ is isometrically isomorphic to $\cB\bowtie_\beta \cH$ for some $(\cG,\cH)$-compatible $\cH$-action $\beta$. 
\end{corollary}

\begin{proof} We first claim that for any $k\in \cK$, there exists a unique $g\in\cG$ and $h\in\cH$ such that $(g,h)\in\cK^{(2)}$ and $k=gh$. To see there exists at least one such pair $(g,h)$, pick any $c\in\cC_k$. By assumption, there exists a unique $b\in\cB$ and $h\in \cH$ such that $c=bu_h$. As $b$ is an element of some fibre $\cB_g$, then $k=gh$. To see that such a pair is unique, suppose $k=g'h'$ for some $g'\in \cG$ and $h'\in \cH$. Then $cu_{h'}^*\in\cC_{kh'^{-1}}=\cC_g\cong \cB_g$. In this case we have $bu_h=c=(cu_{h'}^*) u_{h'}$, so the uniqueness of such decomposition implies $h'=h$ and hence $g'=g$. 

By the internal \ZS product for groupoids \cite[Proposition 3.4]{BPRRW2017}, we conclude $\cK=\cG\bowtie\cH$. The rest of the proof follows from Theorem \ref{thm.internalZS}.
\end{proof}

\section{$C^*$-algebras of Fell bundles}

For a Fell bundle $\cB$ over an \etale groupoid $\cG$, one can define a universal $C^*$-algebra $C^*(\cB)$ with respect to certain $*$-representations of the $*$-algebra $\Gamma_c (\cG;\cB)$ of continuous compactly supported sections of the bundle $\cB$. If we use the standard notation $\cG^v := r_{\cG}\inv(\{v\})$, $\cG_v := s_{\cG}\inv (\{v\})$ for an idempotent $v$ of $\cG$, then the $*$-algebra structure on $\Gamma_c (\cG;\cB)$ is given by
\begin{align}\label{eq:star-alg-structure-of-sections}
    (\sigma\square\tau) (x) &=
    \sum_{y\in \cG^{v}} \sigma(y) \cdot \tau(y\inv x),\text{ and }
    \sigma^* (x) = \sigma(x\inv)^*, 
\end{align}
where $x\in \cG^v$ and where the multiplication $\cdot$ is to be understood in $\cB$. 
Let
\[
    \norm{\sigma}_{I,r}
    :=
     \sup_{v\in\cG\z}\left(
     \sum_{x\in \cG^{v}}
    \norm{\sigma(x)}
    \right)
    \text{ and }
    \norm{\sigma}_{I,s}
    :=
    \sup_{v\in\cG\z}\left(
    \sum_{x\in \cG_{v}}
    \norm{\sigma(x)}\right).
\]
The {\em $I$-norm} on the $*$-algebra $\Gamma_c (\cG;\cB)$ is given by
\[
    \norm{\sigma}_{I}
    :=
    \max(\norm{\sigma}_{I,r}, \norm{\sigma}_{I,s}).
\]
\begin{definition}[see {\cite[Example 4.8]{MW2008}}]
    A $*$-homomorphism $L$ of $\Gamma_c(\cG;\cB)$ into the bounded operators $\mathbb{B}(H)$ on some Hilbert space $H$ is called a {\em  representation} if it is continuous when $\Gamma_c(\cG;\cB)$ has the inductive limit topology and $\mathbb{B}(H)$ the weak operator topology.
    
    Note that $L$ is continuous in this sense if it is {\em $I$-norm decreasing}, i.e.\ $\norm{L(\sigma)}\leq \norm{\sigma}_I$ for all $\sigma\in\Gamma_c (\cG;\cB)$.
\end{definition}

We point out that strict representations of $\cB$ can be integrated to representations of $\Gamma_c(\cG;\cB)$ \cite[Prop. 4.10]{MW2008}, and that the converse holds also: every nondegenerate representation of $\Gamma_c(\cG;\cB)$ is the integrated form of a strict representation \cite[Thm 4.13]{MW2008} (Note that \cite{MW2008} ask for nondegeneracy when they talk about representations; we will explicitly mention when we assume nondegeneracy.).

We define $C^*(\cB)$ to be the completion of $\Gamma_c (\cG;\cB)$ with respect to the universal norm
\[
    \norm{\sigma}
    :=
    \sup\{\norm{L(\sigma)}: L \text{ is an $I$-norm decreasing $*$-representation}\}.
\]
For a more detailed description, one may refer to \cite[Definition 16.25]{ExelFellBundle} for the case when $\cG$ is a discrete group, and \cite{Yamagami1990} or \cite{MW2008} for the case when $\cG$ is an \etale groupoid.

\begin{lemma}\label{lem:blendG}
    For $\sigma\in \Gamma_c (\cG;\cB)$ and $(x,h)\in\cG\bowtie\cH$, define
    \[
        i(\sigma)(x,h)
        :=
        \begin{cases}
        \bigl( \sigma(x), h\bigr)=\Phi(\sigma(x)), & \text{ if  $h=r_{\cH}(h)
        $ is an idempotent,}\\
        \bigl( 0_{x}, h \bigr) & \text{ otherwise},
        \end{cases}
    \]
    where we wrote $0_{x}$ for zero in the Banach space $\cB_{x}$.
    Then $i(\sigma)$ defines an element of $\Gamma_c (\cG\bowtie\cH;\cB\bowtie_\beta\cH)$.
\end{lemma}

\begin{proof}
Let $\cC:= \cB\bowtie_\beta\cH$, with total space $C:= B\tensor*[_{s_{\cB}}]{\times}{_{r_{\cH}}} \cH$, and let $\cK:= \cG\bowtie\cH$. Because $\sigma$ is a section,  $i(\sigma)$ takes values in $C$. Indeed,
\[
    s_{\cB} (\sigma(x))
    =
    s_{\cG} \bigl(p(\sigma(x))\bigr)
    =
    s_{\cG} (x)
    \text{ resp.\ }
    s_{\cB} (0_{x})
    =
    s_{\cG} \bigl(p(0_{x})\bigr) 
    =
    s_{\cG} (x),
\]
both of which equal $r_{\cH}(h)$ since $(x,h)\in \cG\bowtie\cH$. It is, moreover, a section of the bundle $\cC = (C,q)$, as
\[
    q\bigl(i(\sigma)(x,h)\bigr)
    =
    \left\{
    	\begin{array}{ll}
			\bigl( p(\sigma(x)), h\bigr) & \mbox{if } h\in\cH\z \\
			\bigl( p(0_{x}), h \bigr) &\mbox{if } h\notin\cH\z
		\end{array}
	\right\}
    =
    (x,h).
\]

To see that $i(\sigma)$ is continuous, assume that $(x_i,h_i)_i$ is a net in $\cG\bowtie\cH$ which converges to $(x,h)$. If $h$ is not an idempotent, then since $\cH\z$ is closed, there exists $i_0$ such that $h_i \notin \cH\z$ for $i\geq i_0$, so that $    i(\sigma)(x_i,h_i) = \bigl( 0_{{x_i}}, h_i \bigr) .$
By \ref{cond.B-nets} of Definition~\ref{def.USCBdl} for $\cB$, we know that $p(0_{{x_i}})=x_i \to x=p(0_{{x}})$ in $\cG$ implies $0_{{x_i}}\to 0_{{x}}$ in $B$, and so $i(\sigma)(x_i,h_i) \to\bigl( 0_{{x}}, h \bigr) = i(\sigma)(x,h)$, since $h_i\to h$ also.

On the other hand, if $h=v\in \cH\z$, then since $\cH$ is \etale so that $\cH\z$ is also {\em open}, there exists $i_1$ such that for all $i\geq i_1$, we have $h_i=v_i \in \cH\z$ also. Thus, 
\[
    i(\sigma)(x_i,h_i) = \bigl( \sigma(x_i), v_i \bigr) \to \bigl( \sigma(x), v \bigr)
    =
    i(\sigma)(x,h) ,
\]
proving continuity of $i(\sigma)$.

To see that $i(\sigma)$ is compactly supported, note that
\[
    \supp (i(\sigma))
    \subset
    \supp(\sigma)\times s_{\cG}(\supp(\sigma))
    \subset \cG\times\cH.
\]
Since $\supp (\sigma)$ is compact, so is $\supp(\sigma)\times s_{\cG}(\supp(\sigma))$. Since $\cG\bowtie\cH$ is closed in $\cG\times\cH$,
this implies that the closed set $\supp (i(\sigma))$ is contained in a compact subset of $\cG\bowtie\cH$, making it compact.
\end{proof}

\begin{proposition}\label{prop.blendG}
The map
\[
i\colon \Gamma_c (\cG;\cB)\to \Gamma_c (\cG\bowtie\cH;\cB\bowtie_\beta\cH), 
\]
with $i(\sigma)$ as defined in Lemma~\ref{lem:blendG}, is a $*$-algebra homomorphism and
extends to a $*$-homomorphism $i\colon C^*(\cB)\to C^*(\cB\bowtie_\beta \cH)$.
\end{proposition}

\begin{proof}
We have shown in Lemma~\ref{lem:blendG} that $i(\sigma)$ is an element of $\Gamma_c (\cG\bowtie\cH,\cB\bowtie_{\beta}\cH)$, and $\sigma\mapsto i(\sigma)$ is clearly linear.

To see that the map is $*$-preserving, we compute for $(x,h)\in\cG\bowtie\cH$,
\begin{equation}\label{eq:pitstop}
    i(\sigma)^* (x,h)
    \overset{\eqref{eq:star-alg-structure-of-sections}}{=}
    \bigl[i(\sigma)((x,h)\inv)\bigr]^*
    \overset{\eqref{eq:inv-on-ZS-gpd}}{=}
    \bigl[i(\sigma)(h^{-1}\cdot x^{-1}, h^{-1}|_{x^{-1}})\bigr]^*.
\end{equation}
First note that $h^{-1}|_{x^{-1}}\in\cH\z$ if and only if $h\inv\in \cH\z$: we know from \ref{item:restriction-of-unit-is-unit} that the restriction of a unit is a unit, and
\[
    h\inv \overset{\ref{cond.ZS9}}{=} h\inv|_{s_{\cH}(h\inv)}
    =
     h\inv|_{s_{\cG}(x)}
     \overset{\ref{cond.ZS3}}{=} (h\inv|_{x\inv})|_{x}.
\]
So we may first assume that $h^{-1}|_{x^{-1}}, h\inv$ are both elements of $\cH\z$.
By \ref{cond.ZS6} and \ref{cond.ZS8}, we have $h^{-1}|_{x^{-1}}=r_{\cG}(x)$ and $h\inv\cdot x\inv = x\inv$, respectively. Therefore,
\begin{align*}
    i(\sigma)^* (x,h)
    &\overset{\eqref{eq:pitstop}}{=}
     (\sigma(h^{-1}\cdot x^{-1}), r_{\cG}(x))^*
    =
     (\sigma( x^{-1}), r_{\cG}(x))^*
    \\
    &\overset{\ref{cond.C3}}{=}
    \left(\beta_{r_{\cG}(x)\inv}\bigl(\sigma(x^{-1})^*\bigr), r_{\cG}(x)\inv|_{p(\sigma( x^{-1}))\inv}\right) 
    \\
    &\overset{\ref{cond.A3}}{=}
    \left(\sigma( x^{-1})^*, r_{\cG}(x)|_{x}\right) 
    \\
    &\overset{\ref{item:restriction-of-unit-is-unit}}{=}
    \left(\sigma( x^{-1})^*, s_{\cG}(x)\right) 
    \overset{\eqref{eq:star-alg-structure-of-sections}}{=}
    \left(\sigma^*( x), r_{\cH}(h)\right) 
    =i(\sigma^* )(x,h).
\end{align*}

On the other hand, assume $h$ and $h^{-1}|_{x^{-1}}$ are not idempotents.
Since $i(\sigma)^*$ and $i(\sigma^*)$ are both sections, we do not need to keep track of the subscript of the zero-element; instead, we can just compute (using that $\beta$ is linear and hence sends $0$ to $0$)
\begin{align*}
    i(\sigma)^* (x,h)
    &\overset{\eqref{eq:pitstop}}{=}
    (0, (h^{-1}|_{x^{-1}})\inv|_{(h^{-1}\cdot x^{-1})\inv})
    \\
    &\overset{\ref{lm.ZS}(4)}{=}
    (0, (h|_{h\inv\cdot x^{-1}})|_{(h^{-1}\cdot x^{-1})\inv})
    \\
    &\overset{\ref{cond.ZS3}}{\underset{\ref{cond.ZS9}}{=}}
    (0, h)
    =
    i(\sigma^*) (x,h).
\end{align*}
Lastly, we need to see that $\sigma\mapsto i(\sigma)$ is multiplicative, so let $\tau$ be another element of $\Gamma_c (\cG;\cB)$. We compute for $(x,h)\in\cG\bowtie\cH$, using \eqref{eq:star-alg-structure-of-sections}
\begin{align}\label{eq:product-of-tildeL}
    \bigl(i(\sigma)\square  i(\tau) \bigr)(x,h) =
    \sum_{(y,k)\in \cG\bowtie\cH^{r(x,h)}} i(\sigma)(y,k) \bullet   i(\tau)((y,k)\inv (x,h))
    .
\end{align}
Recall from Equation~\eqref{eq:s-and-r-on-SZ} that the range of $(x,h)$ in $(\cG\bowtie\cH)\z$ is given by $r_{\cG}(x)$. Note that, if $(y,k)$ has $k\notin \cH\z$, then $i(\sigma)(y,k) =(0_y, k)$. We get by the definition in  \ref{cond.C2}
\[
    (0,k)\bullet (b,h)= (0, k\vert_{p(b)}h)
    \text{ and }
    (y,k)\bullet (0,h)= (0, k\vert_{p(b)}h).
\]
We see that in Equation~\eqref{eq:product-of-tildeL}, only summands of the form $(y,k)=(y, s_{\cB}(y))$ might not vanish. Using the equality
\[
    (y,s_{\cB}(y))\inv (x,h)
    =
    (y\inv x, h)
\]
in $\cG\bowtie\cH$,
this all in all yields
\begin{align*}
    \bigl(i(\sigma)\square  i(\tau) \bigr)(x,h) &=
    \sum_{y\in \cG^{r(x)}} i(\sigma)(y,s_{\cB}(y))\bullet    i(\tau)(y\inv x, h)
    \\&=
    \begin{cases}
        (0,h) & \text{ if } h\notin \cH\z,
        \\
        \sum_{y\in \cG^{r(x)}} (\sigma(y),s_{\cB}(y))\bullet (\tau(y\inv x), s_{\cB}(x)) & \text{ otherwise}.
    \end{cases}
\end{align*}
On the other hand, $i(\sigma\square\tau) (x,h)$ is also zero if $h\notin \cH\z$, and otherwise
\begin{align*}
    i(\sigma\square\tau) (x,h) &=
        \bigl((\sigma\square\tau) (x), s_{\cB}(x)\bigr)
    =
        \sum_{y\in \cG^{r(x)}} (\sigma(y) \cdot \tau(y\inv x), s_{\cB}(x))
    ,
\end{align*}
which coincides with $\bigl(i(\sigma)\square  i(\tau) \bigr)(x,h)$ by definition, see \ref{cond.C2}.

It remains to show that $i$ extends. Since
$
        \norm{(b,h)}_{\cC}
        =
        \norm{b}_{\cB}
$ for $\cC:=\cB\bowtie_{\beta}\cH$,
    we have
    \begin{align*}
        \norm{i(\sigma)}_{I,s}
        &=
        \sup_{v\in\cK\z}\left(
         \sum_{\varepsilon\in \cK_{v}}
        \norm{i(\sigma)(\varepsilon)}
        \right)
        =
        \sup_{v\in\cK\z}\left(
         \sum_{(x,s_{\cG}(x))\in \cK_{v}}
        \norm{(\sigma(x),v)}
        \right)
        \\
        &\leq
        \sup_{v\in\cK\z}\left(
         \sum_{x\in \cG_{v}}
        \norm{\sigma(x)}
        \right)
        =
        \norm{\sigma}_{I,s},
    \end{align*}
    and similarly
    \begin{align*}
        \norm{i(\sigma)}_{I,r}
        &=
        \sup_{v\in\cK\z}\left(
         \sum_{\varepsilon\in \cK^{v}}
        \norm{i(\sigma)(\varepsilon)}
        \right)
        =
        \sup_{v\in\cK\z}\left(
         \sum_{(x,s_{\cG}(x))\in \cK^{v}}
        \norm{(\sigma(x),s_{\cG}(x))}
        \right)
        \\
        &\leq
        \sup_{v\in\cK\z}\left(
         \sum_{x\in \cG^{v}}
        \norm{\sigma(x)}
        \right)
        =
        \norm{\sigma}_{I,r},
    \end{align*}
    which implies $\norm{i(\sigma)}_I\leq \norm{\sigma}_I.$ Thus, every $I$-norm decreasing representation $L$ of $\Gamma_c (\cK;\cB)$ gives rise to an $I$-norm decreasing representation $L\circ i$ of $\Gamma_c(\cG;\cB)$. In particular,
    \begin{align*}
        \norm{i(\sigma)}_{C^*(\cC)}
        &=
        \sup\{\norm{L(i(f))}: \text{$I$-norm decreasing $*$-rep.\ $L$ of }\Gamma_c (\cK;\cC)\}
        \\
        &\leq
        \sup\{\norm{L'(\sigma)}: \text{$I$-norm decreasing $*$-rep\ $L'$ of }\Gamma_c (\cG;\cB)\}
        ,
    \end{align*}
    i.e.\ $ \norm{i(\sigma)}_{C^*(\cC)} \leq \norm{\sigma}_{C^*(\cB)}$,
    which proves that $i$ extends.
\end{proof}

\begin{definition}\label{def:cov-rep}
    Assume $(\cG,\cH)$ is a matched pair of \etale groupoids and let $\cB=(B,p)$ be a Fell bundle over $\cG$ with a $(\cG,\cH)$-compatible $\cH$-action $\beta$. Let $\cU:= \cG\z=\cH\z.$ A {\em covariant representation of $(\cB, \beta)$} is a quadruple $(\mu, \cU\ast\mscr{H},\hat{\pi}, \hat{M})$ consisting of
    \begin{enumerate}[label=(R\arabic*)]
        \item\label{item:qim} a Radon measure $\mu$ on $\cU$ which is  quasi-invariant with respect to the Haar system of counting measures on $\cG\bowtie\cH$, 
        \item\label{item:BHbdl} a Borel Hilbert bundle $\cU\ast\mscr{H}$ over $\cU$, where we write $\mscr{H}=\{H(v)\}_{v\in\cU}$,
        \item\label{item:star-functor} a Borel $*$-functor $\hat{\pi}\colon \cB \to \End (\cU\ast \mscr{H})$, and
        \item\label{item:gpd-rep} a Borel homomorphism $\hat{M}\colon \cH \to \Iso (\cU\ast\mscr{H})$ such that $\hat{M}_h=(r_{\cH}(h),M_h, s_{\cH}(h))$ for some unitary operator $M_h\colon H(s_{\cH}(h))\to H(r_{\cH}(h))$,
    \end{enumerate}
    such that
    \begin{equation}\label{eq:def:covariance}
        \hat{M}_{h}\hat{\pi}(b) = \hat{\pi} (\beta(h,b))\hat{M}_{h|_{p(b)}}
    \end{equation}
    for all 
    $(h,b)\in\cH \tensor*[_{s_{\cH}}]{\times}{_{r_{\cB}}} \cB$.
\end{definition}

We let $B(\cU\ast\mscr{H})$ denote the Borel sections as defined in 
\cite[Definition~F.1]{Wil2007}, where the interested reader can also find a precise definition of \ref{item:BHbdl}. 
We refer the reader further to \cite[Definition~4.5]{MW2008} for the definition of \ref{item:star-functor} and to \cite[Definition 3.37]{Wil2019} for the definition of  \ref{item:gpd-rep}. We point out that $\hat{\pi}$ being a *-functor in particular allows us to write $\hat{\pi}(b) = (r_{\cB}(b),\pi(b),s_{\cB}(b))$ for some operator $\pi(b)\colon H(s_{\cB}(b))\to H(r_{\cB}(b))$ (so we could have written Equation~\eqref{eq:def:covariance} without the hats).

\begin{example}[see Example~\ref{ex:if-H-is-trivial}]\label{ex:if-H-is-trivial-rep}
   If $\cH=\cG\z$, so that $\cG\bowtie\cH\cong \cG$ and $\cB\bowtie_{\text{triv}}\cH \cong \cB$, then $\mu$ as in Condition~\ref{item:qim} is quasi-invariant with respect to the Haar system of counting measures on $\cG$. Furthermore, since $\hat{M}$ of Condition~\ref{item:gpd-rep} is a homomorphism, it is just (fibre-wise) the identity on $H(v)$ for each $v\in\cH=\cG\z$, so Equation~\eqref{eq:def:covariance} becomes vacuous.  It follows that
   covariant representations of $\cB$ equipped with the trivial $\cG\z$-action are exactly strict representations of $\cB$ in the sense of \cite[Definition 4.9]{MW2008}.
\end{example}

\begin{example}[see Example~\ref{ex.ZSbundle}]\label{ex:if-B-is-trivial-rep}
    Conversely, it was shown in \cite[Appendix B]{MW2008} that, if a Borel $*$-functor $\hat{\pi}$ for the trivial line bundle $\cB(\cG)$ gives rise to a {\em nondegenerate} $*$-representation of $C_c (\cG)$, then $\hat{\pi}$ can be viewed as a unitary representation $\hat{N}$ of $\cG$ by defining
    \[
        \hat{N}_{x} := \left( r_{\cG}(x), \pi(1,x), s_{\cG}(x)\right).
    \]
    Thus, in that case, Definition~\ref{def:cov-rep} boils down to a choice of quasi-invariant measure $\mu$ on $\cG\bowtie\cH$ and two unitary representations $\hat{N}$ of $\cG$ and $\hat{M}$ of $\cH$ on the same Borel Hilbert bundle $\cU\ast\mscr{H}$ satisfying
    \begin{equation}\label{eq:def:covariance-special-case}
        M_{h}N_{x} = N_{h\cdot x}{M}_{h|_{x}}
        , \quad \text{ if } s_{\cH}(h)= r_{\cB}(x).
    \end{equation}
    If we let $K_{(x,h)}:= N_{x}M_{h}$ for $(x,h)\in\cG\bowtie\cH$, then the above equation makes $K$ a homomorphism, so that $(\mu, \cU\ast\mscr{H},\hat{K})$ is a unitary representation of $\cG\bowtie\cH$.
   
\end{example}

As in \cite[F.2]{Wil2007}, we let $L^2(\cU\ast\mscr{H},\mu)$ be the direct integral of the Hilbert bundle, i.e.\ the normed vector space formed by the quotient of
\[
    \cL^2(\cU\ast\mscr{H},\mu):=\left\{ \xi\in B(\cU\ast\mscr{H}) \text{ s.t.\ } v\mapsto \norm{\xi(v)}^2 \text{ is $\mu$-integrable}\right\}
\]
where functions agreeing $\mu$-almost everywhere are being identified. Furthermore, let $\Delta := d\nu / d\nu\inv$ be the Radon-Nykodym derivative of $\nu=\mu\circ\lambda$ and its pushforward $\nu\inv$. Since $\cG$ and $\cH$ are assumed to be second countable and locally compact (so that $\cG\bowtie\cH$ is as well), we know by \cite[Prop.\ 7.9]{Wil2019} that $\Delta\colon \cG\bowtie\cH\to(\mathbb{R}^+, \times)$ can be chosen to be a Borel homomorphism.

For $\varepsilon=(x,h)$ some element of $\cG\bowtie\cH$ and $\sigma$ a section of $\cB\bowtie_\beta\cH$, we will write $\sigma_{\cB}(\varepsilon):=\mathrm{pr}_{\cB}(\sigma(\varepsilon))\in\cB_{x}$  and $\hat{M}_\varepsilon :=\hat{M}_{h}$. Note that we have to be careful with the latter notation; for example, if   $(\varepsilon,\varphi)\in(\cG\bowtie\cH)^{(2)}$ for some $\varphi=(y,k)$, then
    \begin{equation}\label{eq:M-composition}
        \varepsilon\varphi = (x(h\cdot y), h|_y k),
        \text{ so that }
        M_{h|_y} M_\varphi
        =
        M_{h|_y} M_{k}
        =
        M_{h|_y k}
        =
        M_{\varepsilon\varphi}.
    \end{equation}
\begin{theorem}\label{thm:representations-of-ZS-Fellbdl}
    Given a covariant representation $(\mu, \cU\ast\mscr{H},\hat{\pi}, \hat{M})$ of $(\cB, \beta)$,
    define for a section $\sigma\in\Gamma_c (\cG\bowtie\cH;\cB\bowtie_\beta\cH)$, any $\xi\in \cL^2(\cU\ast\mscr{H},\mu) $,  and $v\in \cU$,
    \[
        \bigl(L(\sigma)\xi\bigr) (v)
        =
        \sum_{\substack{\varepsilon\in (\cG\bowtie\cH)^v}}
            \pi ( \sigma_{\cB} (\varepsilon) )M_{\varepsilon}\bigl(\xi(s(\varepsilon))\bigr)
            \Delta(\varepsilon)^{-\frac{1}{2}}
       .
    \]
    Then $L$ is a 
    $I$-norm decreasing $*$-representation of $\Gamma_c(\cG\bowtie\cH;\cB\bowtie_\beta\cH)$ on $L^2(\cU\ast\mscr{H},\mu)$.
\end{theorem}
Similarly to \cite[Remark 4.12]{MW2008},  we emphasize that $L$ need not be nondegenerate and thus not a representation in the sense of \cite[Definition 4.7]{MW2008}. Furthermore, we point out that the modular function $\Delta$ in the formula for $L$ is needed to account for the fact that we defined our involution formula on the $*$-algebra of sections without $\Delta.$
\begin{remark}
    When $\cH=\cG\z$, so that $\cG\bowtie\cH\cong \cG$ and $\cB\bowtie_{\beta}\cH \cong \cB$ (see also Examples~\ref{ex:if-H-is-trivial} and~\ref{ex:if-H-is-trivial-rep}), then the above theorem recovers  \cite[Proposition 4.10.]{MW2008}. 

    Conversely, assume $\cB=\cB(\cG)$ and that we are in the situation of Example~\ref{ex:if-B-is-trivial-rep}, so that the covariant representation $(\mu, \cU\ast\mscr{H},\hat{\pi},\hat{M})$ of $\cB(\cG)$ can be viewed instead as a unitary representation $(\mu, \cU\ast\mscr{H},\hat{K})$ of $\cG\bowtie\cH$, where $K_{(x,h)}:= \pi(1,x)M_{h}$.
    In this case, the above theorem recovers a slightly weakened version of \cite[Proposition 7.12]{Wil2019} in which the assumption of an `almost everywhere unitary representation' in said proposition is made stronger by dropping the word `almost'.
\end{remark}

\begin{proof}[Proof of Theorem~\ref{thm:representations-of-ZS-Fellbdl}]
    For the duration of this proof, let $\cK:=\cG\bowtie\cH$, and let $s$ and $r$ denote its range and source maps. Since $\hat{\pi}$ is a Borel $*$-functor and $\hat{M}$ a Borel homomorphism,
    $\varepsilon\mapsto \hat{\pi}( \sigma_{\cB} (\varepsilon) )\hat{M}_\varepsilon$ is a Borel 
    map from $\cK$ to $\End (\cU\ast\mscr{H})$. As $\Delta$ is Borel also, we thus know that for any $\xi,\zeta\in B(\cU\ast \mscr{H})$, the map
    \[
        F\colon
        \cK
        \to\mathbb{C}, 
        \;
        \varepsilon\mapsto \Biginner{\pi ( \sigma_{\cB} (\varepsilon) )M_{\varepsilon}\bigl(\xi(s (\varepsilon))\bigr)}{\zeta (r (\varepsilon))}
        \Delta (\varepsilon)^{-\frac{1}{2}}
    \]
    is also Borel. 
    We will use an argument from the proof of \cite[Proposition 1.29]{Wil2019} to show that this implies that the map
    \[
        f\colon \cU\to\mathbb{C}, \; v\mapsto
        \sum_{\substack{\varepsilon\in \cK^v}}
            \Biginner{\pi ( \sigma_{\cB} (\varepsilon) )M_{\varepsilon}\bigl(\xi(s (\varepsilon))\bigr)}{\zeta (v)}
            \Delta (\varepsilon)^{-\frac{1}{2}},
    \]
    is Borel also. Since $\cK$ is \etale, we can cover the compact support of $\sigma$ by finitely many open bisections. If we take a continuous partition of unity subordinate to that cover, then we see that we may without loss of generality assume that $\sigma$ is supported in a bisection, say $U\subset\cK$. If we let $t\colon U\to\cK$ be the continuous local inverse of the range map, then $ f(v) = F(t(v))$ is Borel as composition of a Borel and a continuous function.
    
    We next claim that
    it is $\mu$-integrable. To show this, we first compute
    \begin{align*}
        &\int_{\cU} \abs{f(v)}\,  d\mu(v)
        \leq
            \int_{\cK}
                \norm{\pi ( \sigma_{\cB} (\varepsilon) )M_{\varepsilon}\bigl(\xi(s (\varepsilon))\bigr)}
                \norm{\zeta (r (\varepsilon))}
                \Delta (\varepsilon)^{-\frac{1}{2}}
        d\nu (\varepsilon).
    \end{align*}
    Note that the $*$-functor $\hat{\pi}$ is norm-decreasing \cite[Remark 4.6]{MW2008}. Since the codomain of $\hat{M}$ is the isomorphism groupoid, $M_\varepsilon$ is a unitary, so in particular $\norm{M_\varepsilon}=1$. Both combined yield
    \[
        \norm{\pi ( \sigma_{\cB} (\varepsilon) )M_{\varepsilon}\bigl(\xi(s (\varepsilon))\bigr)}
        \leq
        \norm{\sigma_{\cB}(\varepsilon)}\,\norm{\xi(s (\varepsilon))}
        =\norm{\sigma(\varepsilon)}\,\norm{\xi(s (\varepsilon))}.
    \]
    Next, we will use a trick from \cite[Proposition 4.7]{MW2008} resp.\ \cite[Proposition 7.12]{Wil2019}, attributed to Renault: Using Cauchy--Schwarz for $\nu$ in the second of the following inequalities, we get
     \begin{align*}
        \left(\int_{\cU} \abs{f(v)}\,  d\mu(v)\right)^2
        &
        \leq
        \left(
            \int_{\cK}
                \norm{\sigma (\varepsilon)}\,
                \norm{\xi(s(\varepsilon))}\,
                \norm{\zeta (r(\varepsilon))}
                \Delta (\varepsilon)^{-\frac{1}{2}}
        d\nu (\varepsilon)
        \right)^2
        \\
        &\leq
        \left(
            \int_{\cK}
                \norm{\sigma (\varepsilon)}\,
                \norm{\xi(s(\varepsilon))}^2
                \Delta (\varepsilon)\inv
        d\nu (\varepsilon)
        \right)
        \\
        &\qquad\qquad
        \cdot
        \left(
            \int_{\cK}
                \norm{\sigma (\varepsilon)}\,
                \norm{\zeta (r(\varepsilon))}^2
        d\nu (\varepsilon)
        \right).
    \end{align*}
    For the first factor, we compute 
      \begin{align*}
        &\int_{\cU}
        \sum_{\varepsilon \in \cK^v}
            \norm{\sigma(\varepsilon) }\,
            \norm{\xi(s(\varepsilon) )}^2
            \Delta (\varepsilon)\inv
       \,  d\mu (v)
        =
        \int_{\cU}
        \sum_{ \varepsilon' \in \cK_v}
            \norm{\sigma (\varepsilon') }\,
            \norm{\xi(s(\varepsilon'))}^2
       \,  d\mu (v)
        \\
        &\quad \leq
        \int_{\cU}
            \norm{\sigma}_{I,s}
            \norm{\xi(v)}^2
       \,  d\mu (v)
        =
        \norm{\sigma}_{I,s}
        \norm{\xi}_2^2
        .
    \end{align*}
    Similarly one gets for the second factor
      \begin{align*}
        &\int_{\cK}
                \norm{\sigma (\varepsilon)}\,
                \norm{\zeta (r(\varepsilon))}^2
        d\nu (\varepsilon)
        \leq
        \norm{\sigma}_{I,r} \, \norm{\zeta}_{2}^2
        .
    \end{align*}
    Both combined yield
    \begin{align*}
        \left(\int_{\cU} \abs{f(v)}\,  d\mu(v)\right)^2
        \leq
         (\norm{\sigma}_{I,s} \, \norm{\xi}_{2}^2)\,(\norm{\sigma}_{I,r} \, \norm{\zeta}_{2}^2)
         \leq
         \norm{\sigma}_{I}^2 \, \norm{\xi}_{2}^2\, \norm{\zeta}_{2}^2        .
    \end{align*}
    Since
    \begin{align*}
        \inner{L(\sigma)\xi}{\zeta}
        &=
        \int_{\cK}
            \Biginner{\pi ( \sigma_{\cB} (x,h) )M_{h}\bigl(\xi(s_{\cH}(h))\bigr)}{\zeta (v)}
            \Delta(x,h)^{-\frac{1}{2}}
        \,  d\nu(x,h)
        \\
        &=
        \int_{\cU} f(v) \,  d\mu(v),
    \end{align*}
    we have thus proved that $L(\sigma)\xi$ is an element of $\cL^2 (\cU\ast \mscr{H},\mu)$ if $\xi$ is. We point out that this also shows that $L$
    is $I$-norm decreasing.

    We next check that $L$ is multiplicative, so let $\sigma,\tau$ be two sections.
    Using the definition of $L$ twice, we get
    \begin{align*}
        &   (L(\sigma)L(\tau)\xi)
        (v)   
        \\
        &=\sum_{\varepsilon\in \cK^v}
                \pi ( \sigma_{\cB} (\varepsilon) )M_{\varepsilon}\bigl[L(\tau)\xi(s_{\cH}(h))\bigr]
            \Delta(\varepsilon)^{-\frac{1}{2}}
        \\
        &=
        \sum_{\substack{(\varepsilon,\varphi)\in\cK^{(2)} \\ r(\varepsilon)=v}}
            \pi ( \sigma_{\cB} (\varepsilon) )M_{\varepsilon}
            \left[
            \pi ( \tau_{\cB} (\varphi) )M_{\varphi}\bigl(\xi(s(\varphi))\bigr)
            \Delta(\varphi)^{-\frac{1}{2}}
            \right]
            \Delta(\varepsilon)^{-\frac{1}{2}}
        .
    \end{align*}
    If $\varepsilon=(x,h)$ and $\varphi=(y,k)$, then by the covariance condition and the fact that $p(\tau_\cB (y,k)) =y$ since $\tau$ is a section, we have
    \[
        M_h \pi \bigl( \tau_\cB (y,k)\bigr)
        =
        \pi \bigl( \beta(h, \tau_\cB (y,k))\bigr) M_{h|_{y}}
        .
    \]
    Using Equation~\eqref{eq:M-composition}, we arrive at
    \[
        M_{\varepsilon} \pi ( \tau_{\cB} (\varphi) )M_{\varphi}
        =
        \pi \bigl( \beta_{h}( \tau_\cB (\varphi))\bigr)  M_{\varepsilon\varphi}.
    \]
    Since $s(\varphi)=s(\varepsilon\varphi)$, since 
    $\Delta$ is a homomorphism, and since $\hat{\pi}$ is a $*$-functor, we conclude: if $\beta_{\varepsilon}:=\beta_{h}$ for $\varepsilon =(x,h)\in\cK$, then
    \begin{align*}
        &
            (L(\sigma)L(\tau)\xi)
        (v)\\
        &=\sum_{\varepsilon\in\cK^v}
          \sum_{\substack{\varphi\in\cK:\\ s(\varepsilon)=r(\varphi)}}
            \pi \Bigl( \sigma_{\cB} (\varepsilon) \beta_{\varepsilon}(\tau_\cB (\varphi))\Bigr)  M_{\varepsilon\varphi}\,\xi(s(\varepsilon\varphi))
            \Delta(\varepsilon\varphi)^{-\frac{1}{2}}
        \\
        &=
        \sum_{\varepsilon\in\cK^v}
        \sum_{\substack{\varphi\in\cK^v}}
            \pi \Bigl( \sigma_{\cB} (\varepsilon) \beta_{\varepsilon}(\tau_\cB (\varepsilon\inv\varphi))\Bigr)  M_{\varphi}\,\xi(s(\varphi))
            \Delta(\varphi)^{-\frac{1}{2}}
        ,
    \end{align*}
    where we used `left-invariance' of the counting measure in the last step.
    On the other hand,
     \begin{align*}
        (L(\sigma\square \tau)\xi)
        (v)&=
        \sum_{\substack{\substack{\varphi\in \cK^v}}}
            \pi \bigl( (\sigma\square \tau)_{\cB} (\varphi) \bigr)M_{\varphi}\bigl(\xi(s(\varphi))\bigr)
            \Delta(\varphi)^{-\frac{1}{2}}
        .
    \end{align*}
    Using \ref{cond.C2} (the definition of multiplication in $\cB\bowtie_\beta\cH$), we compute for $\varphi\in\cK^v$
    \[
        (\sigma\square \tau)_{\cB } (\varphi)
        =
        \sum_{\varepsilon\in\cK^v}
            \sigma_{\cB} (\varepsilon)
            \beta_{\varepsilon}
            \Bigl(
                \tau_{\cB}
                \bigl(\varepsilon\inv \varphi\bigr)
            \Bigr)
            .
    \]
    Thus, if we exchange the order of the (finite) summation, we see that indeed $L(\sigma)L(\tau)=L(\sigma\square\tau).$
    
    To see that $L$ is $*$-preserving, let us write $\sigma^*_{\cB}:=(\sigma^*)_{\cB}=\mathrm{pr}_{\cB}\circ(\sigma^*)$ and consider
    \[
        \bigl( L(\sigma^*)\zeta\bigr) (v)
        =
        \sum_{\substack{\varepsilon\in (\cG\bowtie\cH)^v}}
            \pi ( \sigma^*_{\cB} (\varepsilon) )M_{\varepsilon}\bigl(\zeta(s(\varepsilon))\bigr)
            \Delta(\varepsilon)^{-\frac{1}{2}}.
    \]
    If $\varepsilon=(x,h)$, then
    $\varepsilon\inv = (h\inv \cdot x\inv, h\inv|_{x\inv}) $ and thus
    \begin{align*}
        \sigma^*_{\cB} (\varepsilon)
        =
        \mathrm{pr}_{\cB} \bigl( \sigma(\varepsilon\inv)^*\bigr)
        &=
         \mathrm{pr}_{\cB}
        \Bigl(\bigl[\sigma_{\cB}(h\inv \cdot x\inv, h\inv|_{x\inv}), h\inv|_{x\inv}\bigr]^*\Bigr)
        \\&\overset{\ref{cond.C3}}{=}
        \beta_{( h\inv|_{x\inv})\inv } \bigl(\sigma_{\cB}(h\inv \cdot x\inv, h\inv|_{x\inv})^*\bigr)\\
        &=
        \beta_{( h\inv|_{x\inv})\inv} \bigl(\sigma_{\cB}(\varepsilon\inv)^*\bigr).
    \end{align*}
    If we let $k = (h\inv|_{x\inv})\inv$, so that $h=k|_{k\inv\cdot x}$, then the above together with the covariance condition gives
    \[
        \pi\bigl( \sigma^*_{\cB} (\varepsilon\inv)\bigr) M_{h}
        =
        \pi\Bigl( 
            \beta_k \bigl(\sigma_{\cB} (\varepsilon\inv)^*\bigr)
        \Bigr)
        M_{k|_{k\inv\cdot x}}
        =
        M_{k} \pi (\sigma_{\cB} (\varepsilon\inv)^*),
    \]
    where we used that $p(\sigma_{\cB} (\varepsilon\inv)^*) = (h\inv \cdot x)\inv =k\inv\cdot x$ because $\sigma$ is a section.
    Since $M_{k}=M_{k\inv}^*= M_{\varepsilon\inv}^*$, since $\pi$ is a $*$-functor, and since $\Delta$ is a homomorphism,  we conclude all in all
    \begin{align*}
        &\inner{\xi}{L(\sigma^*)\zeta}
        =
        \int_{\cK}
        \sum_{\substack{\varepsilon\in (\cG\bowtie\cH)^v}}
        \Biginner{\xi(v)}{
             M_{\varepsilon\inv}^* \pi (\sigma_{\cB} (\varepsilon\inv))^*\bigl(\zeta(s(\varepsilon))\bigr)
        }
        \Delta(\varepsilon)^{-\frac{1}{2}}
        \,  d\mu(v)
        \\&\quad=
        \int_{\cU}
        \sum_{\substack{\varepsilon\in (\cG\bowtie\cH)^v}}
        \Biginner{\pi (\sigma_{\cB} (\varepsilon\inv)) M_{\varepsilon\inv}\bigl(\xi(v)\bigr)}{
             \zeta(s(\varepsilon))
        }
        \Delta(\varepsilon)^{-\frac{1}{2}}
        \,  d\mu(v)
        \\&\quad=
        \int_{\cK}
        \Biginner{\pi (\sigma_{\cB} (\varepsilon\inv)) M_{\varepsilon\inv}\bigl(\xi(s(\varepsilon\inv))\bigr)}{
             \zeta(r(\varepsilon\inv))
        }
        \Delta(\varepsilon\inv)^{\frac{1}{2}}
        \,  d\nu(\varepsilon)
        \\&\quad\overset{(*)}{=}
        \int_{\cK}
        \Biginner{\pi (\sigma_{\cB} (\varphi)) M_{\varphi}\bigl(\xi(s(\varphi))\bigr)}{
             \zeta(r(\varphi))
        }
        \Delta(\varphi)^{\frac{1}{2}}
        \Delta(\varphi)\inv
        \,  d\nu(\varphi)
        \\
        &\quad=\inner{\xi}{L(\sigma)^*\zeta},
    \end{align*}
    where $(*)$ holds by construction of $\Delta$.

    Since we have already seen that $L$ is $I$-norm decreasing, this concludes our proof.
\end{proof}

\begin{definition}
    A covariant representation $(\mu, \cU\ast\mscr{H},\hat{\pi}, \hat{M})$ of $(\cB, \beta)$ induces, by Theorem~\ref{thm:representations-of-ZS-Fellbdl}, a (not necessarily nondegenerate) $*$-representation of $C^*(\cB\bowtie_\beta\cH)$. We denote it by $\hat{\pi}\bowtie\hat{M}$ and call it the {\em integrated form} of $(\mu, \cU\ast\mscr{H},\hat{\pi}, \hat{M})$.
\end{definition}

\begin{theorem}\label{thm.disintegration} Suppose $\cB_u$ is unital for all $u\in \cU$, and let $L$ be an nondegenerate $I$-norm decreasing
$*$-representation of $\Gamma_c(\cG\bowtie \cH; \cB\bowtie_\beta \cH)$. Then there exists a covariant representation $(\mu, \cU\ast\mscr{H},\hat{\pi}, \hat{M})$ of $(\cB,\beta)$ such that $L$ is equivalent 
to the integrated form of $(\mu, \cU\ast\mscr{H},\hat{\pi}, \hat{M})$.
\end{theorem}

\begin{proof} By the disintegration theorem for  representations of Fell bundles \cite[Theorem 4.13]{MW2008}, there exists a strict representation $(\mu, \cU\ast\mscr{H},\hat{\psi})$ of $\cB\bowtie_\beta\cH$ such that for all $\sigma\in \Gamma_c(\cG\bowtie \cH; \cB\bowtie_\beta \cH)$,  $\xi\in \cL^2(\cU\ast\mscr{H},\mu) $,  and $v\in \cU$, 
\begin{equation}\label{eq:disintegrated}
\tilde{L}(\sigma)\xi(v)=\sum_{(x,h)\in (\cG\bowtie \cH)^v} \psi(\sigma(x,h))\xi(s_\cH(h)) \Delta(x,h)^{-\frac{1}{2}},\end{equation}
and $L$ is unitarily equivalent to this integrated form $\tilde{L}$.
For each $b\in \cB$, define $\pi(b)\colon H(s_\cB(b))\to H(r_\cB(b))$ by $\pi(b)=\psi(b,s_\cB(b))$. For each $h\in \cH$, define $M_h\colon H(s_\cH(h))\to H(r_\cH(h))$ by $M_h=\psi(1_{r_\cH(h)},h)$, where $1_{r_\cH(h)}$ is the unit in the unital $C^*$-algebra $\cB_{r_\cH(h)}$. We first prove that $(\mu, \cU\ast\mscr{H},\hat{\pi}, \hat{M})$ is a covariant representation of $(\cB,\beta)$, so let us verify that this quadruple satisfies all the conditions in Definition \ref{def:cov-rep}.

For Conditions~\ref{item:qim} and \ref{item:BHbdl}, since $(\mu, \cU\ast \mscr{H},\psi)$ is a strict representation of  $\cB\bowtie_\beta\cH$, a Fell bundle over $\cG\bowtie \cH$, we automatically have that $\mu$ is quasi-invariant and $\cU\ast\mscr{H}$ is a Borel Hilbert bundle. For any $(a,b)\in\cB^{(2)}$, we have that
\[\pi(a)\pi(b)=\psi(a,s_\cB(a)) \psi(b, s_\cB(b))=\psi(ab,s_\cB(b))=\psi(ab,s_\cB(ab))=\pi(ab).\]
It then follows from the fact that $\psi$ is a Borel $*$-functor that $\pi$ is also a Borel $*$-functor, showing Condition~\ref{item:star-functor}. For each $h\in \cH$, $M_h=\psi(1_{r_\cH(h)},h)$ is a map from $H(s_\cH(h))$ to $H(r_\cH(h))$, and $M_h$ is clearly unitary with inverse $M_{h^{-1}}$.  For any $h,k\in \cH$,
\begin{align*}
    M_h M_k &= \psi(1_{r_\cH(h)},h)\psi(1_{r_\cH(k)},k) \\
    &= \psi(1_{r_\cH(h)}\beta_h(1_{r_\cH(k)}), hk)\\
    &=\psi(1_{r_\cH(h)}, hk)=\psi(1_{r_\cH(hk)}, hk)=M_{hk},
\end{align*}
which shows Condition~\ref{item:gpd-rep}. Finally, we have to check that Equation~\ref{eq:def:covariance} holds, so take any $(h,b)\in\cH \tensor*[_{s_{\cH}}]{\times}{_{r_{\cB}}} \cB$ and compute
\begin{align*}
    M_h \pi(b) &= \psi(1_{r_\cH(h)},h) \psi(b, s_\cB(b)) \\
    &= \psi(\beta(h,b),h|_{p(b)}) \\
    &= \psi(\beta(h,b), s_\cB(\beta(h,b)))\psi(1_{r_\cH(h|_{p(b)})}, h|_{p(b)}) \\
    &= \pi(\beta(h,b)) M_{h|_{p(b)}}.
\end{align*}
Therefore,  $(\mu, \cU\ast\mscr{H},\hat{\pi}, \hat{M})$ is a covariant representation. 

Now, by the definition of $\cB\bowtie_\beta\cH$, \[\sigma(x,h)=(\sigma_\cB(x,h),h)=(\sigma_\cB(x,h),r_\cH(h))(1_{r_\cH(h)},h).\] 
Since $\hat{\psi}$ is a Borel $*$-functor, we thus have that 
\[\psi(\sigma(x,h))=\psi(\sigma_\cB(x,h),r_\cH(h))\psi(1_{r_\cH(h)},h)=\pi(\sigma_\cB(x,h))M_h.\]
Therefore, by Equation~\eqref{eq:disintegrated},
\begin{align*}\tilde{L}(\sigma)\xi(v)&=\sum_{(x,h)\in (\cG\bowtie \cH)^v} \pi(\sigma_\cB(x,h))M_h\xi(s_\cH(h)) \Delta(x,h)^{-\frac{1}{2}}
\\&
= (\hat{\pi}\bowtie\hat{M})(\sigma)\xi (v).\qedhere\end{align*}
This shows that the integrated form of $(\mu, \cU\ast\mscr{H},\hat{\pi}, \hat{M})$ is exactly $\tilde{L}$, which is unitarily equivalent to $L$.
\end{proof}

\section{$C^*$-blend}

In the case of the \ZS product of \etale groupoids and their $C^*$-algebras, it is known that one can find $*$-homomorphisms $i\colon C^*(\cG)\to C^*(\cG\bowtie \cH)$ and $j\colon C^*(\cH)\to C^*(\cB\bowtie \cH)$ such that $(C^*(\cG), C^*(\cH), i, j, C^*(\cG\bowtie \cH))$ is a $C^*$-blend \cite[Theorem 13]{BPRRW2017} in the sense of Exel \cite{Exel2013}. Notice that the groupoid $C^*$-algebra $C^*(\cG)$ is the same as the universal $C^*$-algebra of the groupoid Fell bundle 
$C^*(\mathbb{C}\times \cG)$ for an \etale groupoid $\cG$ (see \cite[Appendix B]{MW2008}).
We have shown in Example \ref{ex.ZSbundle} that the \ZS product of the groupoid Fell bundle $(\mathbb{C}\times \cG)\bowtie \cH$ is the same as the groupoid Fell bundle of the \ZS product $\mathbb{C}\times (\cG\bowtie \cH)$. This alludes to a generalization of the result of Brownlowe et al to \ZS products of Fell bundles.

Recall the definition of $C^*$-blend \cite{Exel2013}:

\pagebreak[3]\begin{definition}\label{def:blend} A {\em $C^*$-blend} is a quintuple $(A_1,A_2,i,j,X)$ where
\begin{enumerate}[label=\textup{(\arabic*)}]
\item $A_1, A_2, X$ are $C^*$-algebras.
\item $i\colon A_1\to \cM(X)$ and $j\colon A_2\to \cM(X)$ are $*$-homomorphisms.
\item\label{item:blend-dense} Define linear maps $i\odot j\colon A_1\otimes_\mathbb{C} A_2 \to \cM(X)$ and $j\odot i\colon A_2\otimes_\mathbb{C} A_1 \to \cM(X)$ on the algebraic tensor products by
\[i\odot j(a\otimes b)=i(a)j(b); j\odot i(b\otimes a)=j(b) i(a).\]
Then the ranges of $i\odot j$ and $j\odot i$ are both dense in $X$. 
\end{enumerate}
\end{definition} 

As pointed out in \cite{Exel2013}, the range of $i\odot j$ is dense if and only if the range of $j\odot i$ is dense, because their ranges are adjoints of each other.

Assume $(\cG,\cH)$ is a matched pair of \etale groupoids and let $\cB=(B,p)$ be a Fell bundle over $\cG$ with a $(\cG,\cH)$-compatible $\cH$-action $\beta$. By Proposition \ref{prop.blendG}, one can build a $*$-homomorphism $i\colon C^*(\cB)\to C^*(\cB\bowtie_\beta \cH)$. We would like to find a $*$-homomorphism $j\colon C^*(\cH)\to C^*(\cB\bowtie_\beta \cH)$.
In order to embed $\cH$ in the \ZS product bundle $\cB\bowtie_\beta\cH$, 
we assume that $\cB_u$ is a unital $C^*$-algebra for all $u\in \cG\z$.

\begin{lemma}\label{lem:Psi}
 Assume that $\cB_u$ is unital for all $u\in \cG\z$. 
 We define for  $f\in C_c (\cH)$ and $(x,h)\in\cG\bowtie\cH,$
\[
    j (f) (x,h)
    =
    \left\{
    \begin{array}{ll}    
        (f(h) 1_x, h)&  \text{if } x=r_{\cH}(h),\\
        (0_{x}, h) & \text{else.}
    \end{array} 
    \right.
\]
Then $j (f)$ is an element of $\Gamma_c (\cG\bowtie\cH; \cB\bowtie_{\beta}\cH)$.
\end{lemma}

\begin{proof}
    Clearly, $j (f)$ is a section. To see that it is compactly supported, note that $\supp (f)$ is contained in $\bigl(s_{\cH}(\supp(f))\times \supp(f)\bigr) \cap \cG\bowtie\cH$, which is compact since $\supp(f)\subset \cH$ is compact and since $s_{\cH}$ is continuous. To see that $j (f)$ is continuous, we use the same argument as in the proof of Proposition~\ref{prop.blendG}, so suppose the net $(x_i, h_i)$ converges to $(x,h)$ in $\cG\bowtie\cH$. If $x=r_{\cH}(h)$, then openness of $\cG\z$ implies that $x_i\in \cG\z$ for large $i$, in which case Condition \ref{cond.B-nets} for the upper semi-continuous bundle $\cB\bowtie_{\beta}\cH$ implies that
    \[
        j (f) (x_i,h_i)
        =
        (f(h_i) 1_{x_i}, h_i)
         \text{ converges to }
        (f(h) 1_x, h)
        =
        j (f) (x,h)
        .
    \]
    If $x\neq r_{\cH}(h)$, then closedness of $\cG\z$ implies that $x_i\not\in \cG\z$ for large $i$, in which case Condition \ref{cond.B-nets} again implies that
    \[
        j (f) (x_i,h_i)
        =
        (0_{x_i}, h_i)
         \text{ converges to }
        (0_{x}, h)
        =
        j (f) (x,h).\qedhere
    \]
\end{proof}

\begin{proposition}\label{prop.blendH}
 Assume that $\cB_u$ is unital for all $u\in \cG\z$. The map $f\mapsto j (f)$ defined in Lemma~\ref{lem:Psi} is a $*$-algebra homomorphism and extends to $j\colon C^*(\cH) \to C^*(\cB\bowtie_\beta \cH)$.
\end{proposition} 

\begin{proof}
    Clearly, $j $ is linear.
    To see that it is multiplicative, we point out that for $(x,h)\in\cK:= \cG\bowtie\cH$ with $v=r_{\cG}(x), u= s_{\cG}(x)$, we have 
    \begin{align*}
        j (f_{1}) \ast j (f_{2}) (x,h)
        &=
        \sum_{(y,k)\in (\cG\bowtie\cH)^{v}} j (f_{1}) (y,k) \,j ( f_{2}) \bigl((y,k)\inv (x,h)\bigr).
    \end{align*}
    Only if $y=v$, the factor $j (f_{1}) (y,k)$ does not vanish, in which case $(y,k)\inv (x,h) = (x,k\inv h)$. The factor $j ( f_{2}) ((y,k)\inv (x,h))$ vanishes unless we also have $x=u$, so that
    \begin{align*}
        j (f_{1}) \ast j (f_{2}) (x,h)
        &=
        \left\{
             \begin{array}{ll}
            \sum_{k\in \cH^{u}}  (f_{1}(k) 1_u,k) \, (f_{2}(k\inv h)1_u,k\inv h),&  \text{if } x=u\\
            (0_{x}, h), & \text{else.}
        \end{array} 
        \right.
    \end{align*}
    By \ref{cond.C2}, we have
    \begin{align*}
        (f_{1}(k) 1_u,k) \, (f_{2}(k\inv h)1_u,k\inv h)
        &=
        (f_{1}(k) 1_u\beta_k(f_{2}(k\inv h)1_u), k|_{p(1_u)} (k\inv h))
        \\&=
        (f_{1}(k)f_{2}(k\inv h)1_u, h)
        ,
    \end{align*}
    so that 
     \begin{align*}
        j (f_{1}) \ast j (f_{2}) (x,h)
        &=
        \left\{
             \begin{array}{ll}   
            \bigl((f_{1}\ast f_{2})(h)\, 1_u, h\bigr),&  \text{if } x=u\\
            (0_{x}, h), & \text{else.}
        \end{array} 
        \right\}
        =
        j (f_{1}\ast f_{2}) (x,h).
    \end{align*}
    To see that $j $ is $*$-preserving,
    recall that
    \begin{align*}
        j  (f)^* (x,h)
        &\overset{\eqref{eq:star-alg-structure-of-sections}}{=}
        \bigl[j  (f) ((x,h)\inv)\bigr]^*
        \overset{\eqref{eq:inv-on-ZS-gpd}}{=}
        \bigl[j  (f) (h^{-1}\cdot x^{-1}, h^{-1}|_{x^{-1}})\bigr]^*.
    \end{align*}
    Note that $x\notin\cG\z$ if and only if $h^{-1}\cdot x^{-1}\notin\cG\z$ (cf.\ to our argument after Equation~\eqref{eq:pitstop}), in which case we  have $j (f)^*(x,h)= (0_x,h)=j (f^*)(x,h)$, simply because $j (f^*)$ and $j (f)^*$ are both sections.
    On the other hand, if $x=r_{\cH}(h)=u\in\cG\z$, then
    \begin{align*}
        j  (f)^* (x,h)
        &=
        \bigl[j  (f) (h^{-1}\cdot u, h^{-1}|_{u})\bigr]^*
        \overset{\ref{item:acting-on-unit-is-unit}}{=}
        \bigl[j  (f) (r_{\cH}(h\inv), h^{-1})\bigr]^*
        \\
        &=
        (f(h\inv) 1_{r_{\cH}(h\inv)}, h\inv)^*
        \overset{\ref{cond.C3}}{=}
        (\overline{f(h\inv)} 1_{h\cdot r_{\cH}(h\inv)}, h)
        \\&\overset{\ref{item:acting-on-unit-is-unit}}{=}
        (\overline{f(h\inv)} 1_{r_{\cH}(h)}, h)
        =
        j  (f^* )(x,h).
    \end{align*}
    It remains to see that $j$ extends. 
    Since
    \[
        \norm{(f(h)1_{v}, h)}
        =
        \abs{f(h)}\,
        \norm{1_{v}}
        =
        \abs{f(h)},
    \]
    one can argue, {\em mutatis mutandis}, as in Proposition~\ref{prop.blendG} that 
    $\norm{j(f)}_I\leq \norm{f}_I$ and thus $ \norm{j(f)}_{C^*(\cC)} \leq \norm{f}_{C^*(\cH)}$,
    which proves that $j$ extends.
\end{proof}

\begin{theorem}\label{thm.blend} The quintuple $(C^*(\cB), C^*(\cH), i, j, C^*(\cB\bowtie_\beta \cH))$ is a $C^*$-blend, where $i$ is given by Proposition \ref{prop.blendG}, and $j$ is given by Proposition \ref{prop.blendH}. 
\end{theorem}

By Propositions \ref{prop.blendG} and \ref{prop.blendH}, $i$ and $j$ are $*$-homomorphisms to  $C^* (\cB\bowtie_{\beta}\cH)
$, and we want to show that they satisfy \ref{item:blend-dense} of Definition~\ref{def:blend}. For the proof, we will need the following helpful lemma.
\begin{lemma}\label{lem:dense-sections}
    Suppose $\cC=(C,q)$ is an upper semi-continuous Banach bundle over some locally compact Hausdorff space $X$, and let $\Gamma \subset \Gamma_0 (X;\cC)$ and $\mathfrak{X}\subset C_0 (X)$ be subspaces such that
    \begin{enumerate}[label=\textup{(\arabic*)}]
        \item\label{assu.Xdense} $\mathfrak{X}$ is uniformly dense,
        \item\label{assu.module} if $\mathfrak{x}\in\mathfrak{X}$ and $\mathfrak{s}\in\Gamma$, then their \textup(pointwise\textup) product $\mathfrak{x}\mf{s}$ is in $\Gamma$, and
        \item\label{assu.ptwdense} for each $x\in X$, the set $\Gamma(x):=\{\mf{s}(x)\,\vert\, \mf{s}\in\Gamma\}$ is dense in $\cC_x$.
    \end{enumerate}
    Then $\Gamma$ is uniformly dense in $\Gamma_0 (X;\cC)$.
\end{lemma}
The authors would like to thank Dana Williams for pointing them to \cite[Prop.\ C.24]{Wil2007}.
\begin{proof}
    Let $\mf{t}\in\Gamma_0 (X;\cC)$ and $\epsilon >0$ be arbitrary. 
    Because of Assumption~\ref{assu.ptwdense} and upper semi-continuity of $\cC$, we can use the proof of \cite[Prop.\ C.24]{Wil2007} to find $\mf{s}_{1},\ldots, \mf{s}_n\in\Gamma$ and $\rho_1,\ldots, \rho_n \in C_0 (X)$ such that
    \begin{equation*}
        \norm{\mf{t}-\sum_{i=1}^{n} \rho_{i}\mf{s}_i}_{\infty} < 
        \frac{\epsilon}{2}.
    \end{equation*}

    By Assumption~\ref{assu.Xdense}, we can find $\psi_1,\ldots,\psi_n\in \mf{X}$ such that
    \[
    \norm{\rho_i - \psi_i}_{\infty} <
        \frac{\epsilon}{2n\norm{\mf{s}_i}_{\infty} +1}
    \]
    for each $1\leq i\leq n$. This yields
    \begin{align*}
        \norm{ \mf{t} - \sum_{i=1}^{n} \psi_{i}\mf{s}_i}_{\infty}
        &\leq
        \norm{ \mf{t} - \sum_{i=1}^{n} \rho_{i}\mf{s}_i}_{\infty} +
        \sum_{i=1}^{n}  \norm{\rho_i - \psi_{i}}_{\infty} \norm{\mf{s}_i}_{\infty}
        \\
        &< 
        \frac{\epsilon}{2}+ \sum_{i=1}^{n}  
        \frac{\epsilon}{2n\norm{\mf{s}_i}_{\infty} +1} \norm{\mf{s}_i}_{\infty}
      <
        \epsilon.
    \end{align*}
    By Assumption~\ref{assu.module} and since $\Gamma$ is a subspace, we have $\sum_{i}\psi_{i}\mf{s}_i\in \Gamma$, so that we have approximated the arbitrary element $\mf{t}$ by an element of $\Gamma$.
\end{proof}

\begin{lemma}\label{lem:ioj-range-uniformly-dense}
    In the setting of Theorem~\ref{thm.blend}, the range of $i\odot j$ is contained in $\Gamma_c (\cG\bowtie\cH;\cB\bowtie_{\beta}\cH)$ and it is uniformly dense in $\Gamma_0 (\cG\bowtie\cH;\cB\bowtie_{\beta}\cH)$.
\end{lemma}

\begin{proof}
    Let $\cC=(C,q)$ be the Fell bundle $\cB\bowtie_{\beta}\cH$, and $\cK:= \cG\bowtie\cH$. Let us first show that the range of $i\odot j$ is contained in $\Gamma_c (\cK;\cC)$, so take  $\sigma\in \Gamma_c (\cG;\cB)$ and $f\in C_c (\cH)$. If $\varepsilon\in\cK$, then by definition of the product in $\Gamma_c (\cK;\cC)$ (see Equation~\eqref{eq:star-alg-structure-of-sections}), we have
\[
    i\odot j (\sigma\otimes_{\mathbb{C}} f) (\varepsilon)
    =
    \sum_{\varphi: r_{\cK}(\varphi)=r_{\cK}(\varepsilon)} i(\sigma)(\varphi) \bullet j (f)(\varphi\inv \varepsilon),
\]
where $\bullet$ refers to multiplication in $\cC$. Only if $\varphi$ is of the form $(y,s_{\cG}(y))$ for some $y\in\cG$ does $i(\sigma)(\varphi)$ not necessarily vanish, in which case $\varphi\inv = (y\inv, r_{\cG}(y))$ and  $\varphi\inv \varepsilon = (y\inv x, h)$ where $\varepsilon=(x,h).$ Similarly, $j (f)(y\inv x, h)$ vanishes if $y\inv x \notin \cG\z,$ i.e.\ $y=x$. We conclude 
\[
    i\odot j (\sigma\otimes_{\mathbb{C}} f) (x,h)
    =
    i(\sigma)(x,s_{\cG}(x)) \bullet j (f)(r_{\cH}(h), h).
\]
Let $v=s_{\cG}(x)=r_{\cH}(h)$. Using the definition of $i$ resp.\ $j,$ we see that
\[
   i(\sigma)(x,v) =(\sigma(x),v)
   \text{ and }
   j (f)(v, h)    = (f(h)1_v, h).
\]
By definition of $\bullet$ (see \ref{cond.C2}) and since $\beta_v=\mathrm{id}$ since $v\in\cH\z$, we have
\begin{equation}\label{eq:iodotj-range}
    i\odot j (\sigma\otimes_{\mathbb{C}} f) (x,h)
    =
    (\sigma(x),v) \bullet (f(h)1_v, h)
    =
    ( \sigma(x)f(h), h)
\end{equation}
where we used that $f(h)$ is just a scalar. Thus, $i\odot j (\sigma\otimes_{\mathbb{C}} f)$ is clearly continuous and, by construction, a section. The above shows furthermore that 
\[
    \supp \bigl( i\odot j (\sigma\otimes_{\mathbb{C}} f)\bigr)
    = \bigl( \supp(\sigma)\times\supp(f)\bigr) \cap \cK,
\]
which is compact since $\mc{K}$ is closed in $\cG\times\cH$. This implies that the range of $i\odot j$ is a subspace of $\Gamma_{c}(\cK;\cC)$. 

To see that the range is dense, we will employ Lemma~\ref{lem:dense-sections}. Let $\Gamma$ be the range of $i\odot j$ and let $\mf{X}$ be the linear span of those functions $F$ in $C_0 (\cK)$ such that $F(x,h)=f_{1}(x)f_{2}(h)$ for some $f_{1}\in C_0 (\cG)$ and $f_{2}\in C_0 (\cH)$; we will write $f_{1}\bowtie f_{2}:= F$ for the duration of this proof. Clearly, $\mf{X}$ is a $*$-subalgebra of $C_0 (\cK)$ that separates points, and for each $(x,h)\in\cK$, we may find $f_{1}, f_{2}$ with $f_{1}(x)\neq 0 \neq f_{2} (h)$, i.e.\ $f_{1}\bowtie f_{2}(x,h)\neq 0$. Thus, by the Stone--Weierstrass Theorem, $\mf{X}$ is dense in $C_{0} (\cK)$, i.e.\ Assumption~\ref{assu.Xdense} of Lemma~\ref{lem:dense-sections} holds.

Furthermore, we can rewrite the pointwise product of any $f_{1}\bowtie f_{2}$ with elements in the range of $i\odot j$ as follows:
\[
    f_{1}\bowtie f_{2} [i\odot j (\sigma\otimes_{\mathbb{C}} f)] =
    i\odot j ((f_{1}\sigma)\otimes_{\mathbb{C}} (f_{2}f)),
\]
since $\Gamma_{c}(\cG;\cB)$ is a $C_{0}(\cG)$-module and $C_{c}(\cH)$ is a $C_{0}(\cH)$-module. This shows that the element $f_{1}\bowtie f_{2} [i\odot j (\sigma\otimes_{\mathbb{C}} f)] $ is in the range of $i\odot j$, i.e.\ $\Gamma$ satisfies Assumption~\ref{assu.module} of Lemma~\ref{lem:dense-sections}.

Next fix $(x,h)\in\cK$, any element $b$ in $\cB_{x}\cong \cC_{(x,h)}$, and $\epsilon>0$. Since $\cB$ has enough continuous cross-sections (see \cite[Theorem~12]{FellBundleBook}), i.e.\ $\Gamma_{c}(\cG;\cB)(x)=\cB_{x}$, we can find $\sigma\in \Gamma_{c}(\cG;\cB)$ with $\norm{b - \sigma(x)}< \epsilon$. If $f\in C_c (\cH)$ with $f(h)=1$, then by Equation~\eqref{eq:iodotj-range},
\[
    \norm{(b,h) - i\odot j (\sigma\otimes_{\mathbb{C}} f) (x,h)}
    =
    \norm{ (b- \sigma(x), h)} < \epsilon,
\]
which proves that $\Gamma(x,h)$ is dense in $\cC_{(x,h)}$, i.e.\ Assumption~\ref{assu.ptwdense} of Lemma~\ref{lem:dense-sections} is also satisfied. It follows that the range $\Gamma$ of $i\odot j$ is uniformly dense in $\Gamma_0 (\cK;\cC)$.
\end{proof}

\begin{proof}[Proof of Theorem~\ref{thm.blend}]
We only have to show that $i\odot j$ has dense range, so fix an arbitrary $\mf{t} \in \Gamma_{c}(\cK;\cC)$, where we again write $\cC$ for the bundle $\cB\bowtie_{\beta}\cH$ and $\cK$ for $\cG\bowtie\cH$. By a standard `partition of unity' argument, using the compact support of $\mf{t}$, we can 
without loss of generality assume that $\supp'(\mf{t})$ is contained in a basic open set $W$. 

Note that the topology of $\cK$ has a basis of open sets consisting of $U\bowtie V := (U\times V)\cap \cK$, where $U\subset \cG$, $V\subset \cH$ are basic open sets. In particular, we can assume that $W=U\bowtie V$, where $U$ and $V$ are small bisections, i.e.\ bisections which are precompact and whose closure is contained in another open bisection, say $\overline{U}\subset U'$ and $\overline{V}\subset V'$ (See \cite[Lemma 5.1.]{JC:structure}).

Fix $\epsilon>0$. By Lemma~\ref{lem:ioj-range-uniformly-dense}, we know that there exists $\mf{s}$ in the range of $i\odot j$ such that $\norm{\mf{t}-\mf{s}}_{\infty}<\epsilon.$ Let $f_{1}\in C_c (\cG)$ be a $[0,1]$-valued function such that $f_{1}|_{U}\equiv 1$ and $f_{1}$ vanishes off of $U'$. Similarly, let $f_{2}\in C_c (\cH)$ be such that $f_{2}|_{V}\equiv 1, $ and $f_{2}$ vanishes off of $V'$. We have seen in the proof of Lemma~\ref{lem:ioj-range-uniformly-dense} that $\mf{s}'\colon (x,h)\mapsto f_{1}(x)f_{2}(h)\mf{s}(x,h)$ is also in the range of $i\odot j$. Since $\mf{t}$ vanishes off of $U\bowtie V$ and since $f_{1}, f_{2}$ are $[0,1]$-valued, we see that
\begin{align*}
  \norm{\mf{t}-\mf{s}'}_{\infty}
  \leq
  \max\left( \sup_{\varepsilon\in U\bowtie V}\norm{(\mf{t} -\mf{s}')(\varepsilon)}, \sup_{\varepsilon\in 
  \cK\setminus (U\bowtie V)}\norm{\mf{s}(\varepsilon)}\right).
\end{align*}
Since $f_{1}|_{U}\equiv 1$ and $f_{2}|_{V}\equiv 1, $ the right-hand side is exactly $\norm{\mf{t}-\mf{s}}_{\infty}$, which is smaller than $\epsilon$ by choice of $\mf{s}$.
Note that $\mf{t}-\mf{s}'$ is supported in the open bisection $ U'\bowtie V'$. We may thus use the computation in \cite[Lemma 4.4]{BPRRW2017} to conclude that $\norm{\mf{t}-\mf{s}'}_{I}=\norm{\mf{t}-\mf{s}'}_{\infty} < \epsilon$. Since the $I$-norm dominates the full $C^*$-norm, this proves that~$\mf{s}'$, an element in the range of~$i\odot j$, approximates~$\mf{t}$ in $C^*(\cC)$.
\end{proof}

\renewcommand{\cG}{\Gamma}
\renewcommand{\cH}{\Lambda}
\renewcommand{\cK}{\Omega}

\section{Embedding of $C^*(\cB)$ into $C^*(\cB\bowtie_\beta\cH)$}

It is a well known fact that for a $C^*$-dynamical system $(A,\cH,\alpha)$ where $\cH$ is a discrete group, there exists an injective $*$-homomorphism $i\colon A\to A\rtimes_\alpha \cH$ that embeds $A$ inside the crossed product $C^*$-algebra $A\rtimes_\alpha \cH$, since $A$ embeds injectively inside the dense $*$-subalgebra $\Gamma_c(\cH,A)$ via $a\mapsto a u_e$.
We note that the crossed product 
can be viewed as the $C^*$-algebra of a \ZS product of Fell bundle, as illustrated in the following example. 

\begin{example} Let $A$ be a $C^*$-algebra, $\cH$ a discrete 
group, and $(A,\cH,\alpha)$ a $C^*$-dynamical system in the classical sense. If we think of~$A$ as a Fell bundle ${\mathscr{A}}$ over the trivial group $\{e\}$, then~$A=C^*({\mathscr{A}})$. Since the $\cH$-action $\alpha$ satisfies all conditions in Definition~\ref{df.beta}, we can form the \ZS product $\mathscr{A}\bowtie_\alpha \cH$, a Fell bundle over $\{e\}\bowtie\cH \cong \cH$. One can verify that its $C^*$-algebra is canonically isomorphic to the crossed product $C^*$-algebra $A\rtimes_\alpha \cH$. In particular, when $A$ is unital, this implies that the inclusion map $i\colon A=C^*(\mathscr{A})\to C^*(\mathscr{A}\bowtie_\alpha \cH)$ defined in Proposition~\ref{prop.blendG} is injective. 
\end{example} 

The above example motivates the question whether the inclusion map $i\colon C^*(\cB)\to C^*(\cB\bowtie_\beta \cH)$ from Proposition \ref{prop.blendG} is always injective when $\cH$ is discrete. Throughout this section, we assume that $\cG$ and $\cH$ are discrete groups, so that $\cU= \cG\z=\cH\z$ is merely $\{e\}$. We further assume that the $C^*$-algebra $\cB_e$ is unital, with its unit denoted by $1$.

In this setting, a covariant representation $(\mu, \cU\ast\mscr{H},\hat{\pi}, \hat{M})$ of $(\cB, \beta)$ collapses to a $*$-representation $\pi\colon\cB\to \mathbb{B}(H)$ and a unitary representation $M\colon \cH \to \mathbb{U}(H)$ on some Hilbert space $H$ which are covariant in the sense of Equation~\eqref{eq:def:covariance}.

\begin{proposition}\label{prop:amplification}
Assume $(\cG,\cH)$ is a matched pair of discrete groups and let $\cB=(B,p)$ be a Fell bundle over $\cG$ with a $(\cG,\cH)$-compatible $\cH$-action $\beta$. 
If $\pi$ is a strict representation of $\cB$ on $H$, let $\widetilde{H} := \ell^2 (\cG\bowtie\cH)\otimes H$ and for $b\in\cB$, $(x,h)\in\cG\bowtie\cH$, and $\xi\in H$, define
\[ 
  \Pi_0(b) \bigl(\delta_{(x,h)}\otimes \xi\bigr)= 
  \delta_{(p(b)x,h)}
  \otimes 
  \pi\bigl( \beta(h\inv|_{(p(b)x)\inv},b) \bigr)
    (
    \xi
    ).
\]
Then \label{item:Pi0-extends} $\Pi_0(b)$ extends to a bounded linear map $\Pi(b)\in \mathbb{B}(\widetilde{H}),$ and $\Pi\colon\cB\to \mathbb{B}(\widetilde{H})$ is a $*$-representation.
\end{proposition} 

\begin{proof}  
    To see that $\Pi_0 (b)$ extends, we compute for $f\in c_{00}(\cG\bowtie\cH, H)\subset \widetilde{H}$:
    \begin{align*}
        \norm{\Pi_0(b) (f)}_2^2
        &=
        \sum_{(x,h)\in \cG\bowtie\cH} \abs{\Pi_0(b) (f) (x,h)}
        \\
        &=
        \sum_{(x,h)\in \cG\bowtie\cH} \abs{\pi\bigl( \beta(h\inv|_{x\inv},b) \bigr)\bigl[ f(p(b)\inv x, h) \bigr]}^2
        \\
        &\leq
        \sum_{(x,h)\in \cG\bowtie\cH}
        \norm{\pi\bigl( \beta(h\inv|_{x\inv},b) \bigr)}^2
        \,
        \abs{f(p(b)\inv x, h)}^2.
    \end{align*}
    Since $\pi$ is a representation of $\cB$ on $H$, it is norm decreasing. Furthermore, each $\beta_k$ is isometric by Corollary~\ref{cor.iso}, so we have
    \[
        \norm{\pi\bigl( \beta(h\inv|_{x\inv},b) \bigr)}
        \leq
        \norm{\beta(h\inv|_{x\inv},b)}
        =
        \norm{b}.
    \]
Thus,
    \begin{align*}
        \norm{\Pi_0(b) (f)}_2^2
        &\leq
        \norm{b}^2
        \,
        \sum_{(x,h)\in \cG\bowtie\cH}
            \abs{f(p(b)\inv x, h)}^2
        \\
        &\overset{(*)}{=}
        \norm{b}^2
    \,
        \sum_{(y,h)\in \cG\bowtie\cH}
            \abs{f(y, h)}^2
        =
        \norm{b}^2
        \,
        \norm{f}_2^2
        ,
    \end{align*}
    where $(*)$ follows from the left-invariance of counting measure on the discrete group $\cG\bowtie\cH$. This shows that $\Pi_0(b)$ extends to a map $\Pi(b)\in\mathbb{B}(\widetilde{H})$.
    
    Let us now prove that $\Pi$ is a $*$-representation. To see that $\Pi$ is fibrewise linear, let $b, c \in \cB$ be in the same fibre, so that $y:=p(b)=p(c)=p(\lambda b+c)$ for any $\lambda\in\mathbb{C}$. Then for $f\in c_{00},$
    \[
        \Pi_0 (\lambda b + c)  (f) (x,h)
        =
        \pi\bigl( \beta(h\inv|_{x\inv},\lambda b+c) \bigr)\bigl[ f(y\inv x, h) \bigr].
    \]
    Since $\beta$ was assumed to be fibrewise linear, and since $\pi$ is a $*$-homomorphism, we have
    \[
        \pi\bigl( \beta(h\inv|_{x\inv},\lambda b+c) \bigr)
        =
        \lambda \pi\bigl( \beta(h\inv|_{x\inv},b) \bigr)
        +
        \pi\bigl( \beta(h\inv|_{x\inv},c) \bigr),
    \]
    which proves $\Pi (\lambda b + c)=\lambda\Pi (b)+\Pi (c)$. Similarly, to see that $\Pi$ is fibrewise multiplicative, let $(b,c)\in\cB^{(2)}$. Then $p(b\cdot c)=p(b)p(c)$  by \ref{cond.F1} and $\beta(h\inv|_{x\inv},b\cdot c)=\beta(h\inv|_{x\inv},b)\beta(h\inv|_{x\inv p(b)},c)$ by \ref{cond.A4}. Thus  
    \begin{align*}
        \pi\bigl( \beta(h\inv|_{x\inv},b\cdot c) \bigr)
        &=
        \pi\bigl(\beta(h\inv|_{x\inv},b)\beta(h\inv|_{x\inv p(b)},c)\bigr)
        \\&=
        \pi\bigl(\beta(h\inv|_{x\inv},b)\bigr)\pi\bigl(\beta(h\inv|_{x\inv p(b)},c)\bigr),
    \end{align*}
    so that
    \begin{align*}
        &\Pi_0 (b\cdot c)  (f) (x,h)
        =
        \pi\bigl( \beta(h\inv|_{x\inv},b\cdot c) \bigr)\bigl[ f(p(b\cdot c)\inv x, h) \bigr]\\
        &=
        \pi\bigl(\beta(h\inv|_{x\inv},b)\bigr)\pi\bigl(\beta(h\inv|_{x\inv p(b)},c)\bigr)\bigl[ f(p(c)\inv p(b)\inv x, h) \bigr]
        \\
        &=
        \pi\bigl(\beta(h\inv|_{x\inv},b)\bigr)\bigl[\Pi_0 (c)  (f) (p(b)\inv x,h) \bigr]
        =
        \Pi_0 (b)  (\Pi_0 (c)  (f)) (x,h)
        ,
    \end{align*}
    proving that $\Pi(b\cdot c) = \Pi (b) \Pi (c).$ Lastly, to see that $\Pi$ is fibrewise $*$-preserving, note that $p(b^*)=p(b)\inv$  by \ref{cond.F5} and also $\beta(h\inv|_{y\inv},b^*)=\beta(h\inv|_{y\inv p(b)\inv},b)^*$ by \ref{cond.A5}. Thus
    \begin{align*}
         \pi\bigl( \beta(h\inv|_{x\inv},b^*) \bigr)=
          \pi\bigl( \beta(h\inv|_{(p(b)x)\inv},b)\bigr)^* ,
    \end{align*}
    so with  $g\in c_{00}(\cG\bowtie\cH, H)$, and $(y,h)\in \cG\bowtie\cH$, we have
    \begin{align*}
        \Pi_0 (b^*)  (g) (y,h)
        &=
        \pi\bigl( \beta(h\inv|_{y\inv},b^*) \bigr)\bigl[ g(p(b^*)\inv y, h) \bigr]
        \\
        &=
        \pi\bigl( \beta(h\inv|_{(p(b)y)\inv},b) \bigr)^*\bigl[ g(p(b) y, h) \bigr].
    \end{align*}
    Therefore, if $f\in c_{00}(\cG\bowtie\cH, H)$, then
    \begin{align*}
        &\Biginner{\Pi_0 (b^*) (g)}{f}^{\widetilde{H}}
        =
        \sum_{(y,h)\in (\cG\bowtie\cH)^v}
        \Biginner{\Pi_0 (b^*) (g) (y,h)}{f (y,h)}^{H}
        \\
        &=
        \sum_{(y,h)\in \cG\bowtie\cH}
        \Biginner{\pi\bigl( \beta(h\inv|_{(p(b)y)\inv},b) \bigr)^*\bigl[ g(p(b) y, h) \bigr]}{f (y,h)}^{H}
        \\
        &=
        \sum_{(y,h)\in \cG\bowtie\cH}
        \Biginner{g(p(b) y, h)}{\pi\bigl( \beta(h\inv|_{(p(b)y)\inv},b) \bigr) [f(y,h)]}^{H}
        \\
        &\overset{(*)}{=}
        \sum_{(x,h)\in \cG\bowtie\cH}
        \Biginner{g(x, h)}{\pi\bigl( \beta(h\inv|_{x\inv},b) \bigr) [f(p(b)\inv x,h)]}^{H}
        \\
        &=
        \Biginner{g}{\Pi_0 (b) f}^{\widetilde{H}}
        =
        \Biginner{\Pi_0 (b)^* (g)}{ f}^{\widetilde{H}},
    \end{align*} where $(*)$ follows, again, from the left-invariance of counting measure. This proves that $\Pi(b^*)=\Pi(b)^*,$ and all in all that $\Pi$ is a $*$-homomorphism.
\end{proof} 

\begin{definition}
    In the setting of Proposition~\ref{prop:amplification}, the properties we have proved yield that $\Pi$ is another strict representation of $\cB$. We will call it the {\em twisted amplification of $\pi$.}
\end{definition}

For a fixed $h\in \cH$, if $\xi\in H$
and $(x,k)\in \cG\bowtie\cH$, we define
\[
    M_{h} (\delta_{(x,k)}\otimes \xi)
    =\delta_{(e,h)(x,k)}
    \otimes \xi.
\]

\begin{lemma}\label{lm.unitaryM}
   The map $M_{h}$ 
   extends to a unitary map on $\widetilde{H}=\ell^2 (\cG\bowtie\cH)\otimes H $. Moreover, $M\colon\cH\to \mathbb{U}(\widetilde{H})$ is a unitary representation of $\cH$. 
\end{lemma}

\begin{proof}
    To see that $M_{h}$ extends, we compute for $f\in c_{00}(\cG\bowtie\cH, H)\subset \widetilde{H}$
    \begin{align*}
        \norm{M_{h} (f)}_{2}^2
        &=
        \sum_{(x,k)\in \cG\bowtie\cH}
        \norm{ f ( (e,h)^{-1} (x, k)}^2
        \\
        &\overset{(*)}{=}
        \sum_{(y,l)\in \cG\bowtie\cH}
        \norm{f ( y, l) }^2
        =
        \norm{f}_{2}^2,
    \end{align*}
    where $(*)$ holds because the counting measure on $\cG\bowtie\cH$ is left-invariant.   Therefore, $M_h$ extends to an isometric linear map on $\ell^2 (\cG\bowtie\cH)\otimes H $. For any $g,h\in\cH$ and $f$ as above, one can compute
    \begin{align*}
    M_g M_h (f)(x,k) &= M_h(f)((e,g)^{-1}(x,k)) \\
    &= f((e,h)^{-1}(e,g)^{-1}(x,k)) \\
    &= f((e,gh)^{-1}(x,k)) \\
    &= M_{gh}(f)(x,k).
    \end{align*}
Therefore, $M_g M_h=M_{gh}$ and $M\colon\cH\to \mathbb{U}(\ell^2 (\cG\bowtie\cH)\otimes H )$ is a homomorphism. We have that $M_{h\inv} M_h=M_h M_{h\inv}=M_e=I$, and thus each $M_h$ is unitary and $M$ is a unitary representation.
\end{proof}

\begin{lemma}\label{lm.M}
  Let $\pi$ and its twisted amplification $\Pi$ be as in Proposition~\ref{prop:amplification} and $M$ be as in Lemma \ref{lm.unitaryM}. Then,
  \[
        M_{h}\Pi(b) = \Pi (\beta(h,b))M_{h|_{p(b)}}
  \]
    for all $h\in\cH$ and $b\in \cB$. In particular, $(\Pi,M)$ is a covariant representation of $(\cB, \beta)$ in the sense of Definition~\ref{def:cov-rep}. 
\end{lemma}

\begin{proof} It only remains to check the covariance condition in Equation~\ref{eq:def:covariance}, so let $p(b)=y\in \cG$ and take any $f\in c_{00}(\cG\bowtie\cH)\otimes H $. By definition,
\begin{align*}
(M_h\Pi(b)f)(x,k)
&= (\Pi(b)f)(h\inv\cdot x, h\inv|_x k) \\
&= \pi(\beta((h\inv|_x k)\inv|_{(h\inv\cdot x)\inv},b))f(y\inv(h\inv\cdot x), h\inv|_x k),
\end{align*}
and
\begin{align*}
(\Pi(\beta(h,b))M_{h|_y}f)(x,k) 
&= \pi(\beta(k^{-1}|_{x\inv}, \beta(h,b))) (M_{h|_y}f)(p(\beta(h,b))^{-1} x,k) \\
&= \pi(\beta(k^{-1}|_{x\inv}, \beta(h,b))) f((e,h|_y)\inv (p(\beta(h,b))^{-1} x,k)).
\end{align*}
By ~\ref{cond.A2},
\[\beta(k^{-1}|_{x\inv}, \beta(h,b)) = \beta(k^{-1}|_{x\inv} h,b)).\]
On the other hand, 
\begin{align*}
(h\inv|_x k)\inv|_{(h\inv\cdot x)\inv} &=  k\inv|_{(h\inv|_x)\inv\cdot (h\inv\cdot x)\inv}(h\inv|_x)\inv|_{(h\inv\cdot x)\inv} \\
 &=  k\inv|_{(h\inv|_x)\inv\cdot (h\inv|_x \cdot x\inv)}(h|_{h\inv\cdot x})|_{(h\inv\cdot x)\inv} \\
 &= k\inv|_{x\inv} h.
\end{align*}
Therefore,
\[\pi(\beta((h\inv|_x k)\inv|_{(h\inv\cdot x)\inv},b))=\pi(\beta(k^{-1}|_{x\inv}, \beta(h,b))).\]
Now by ~\ref{cond.A1}, $p(\beta(h,b))=h\cdot p(b)=h\cdot y$. We can compute that 
\begin{align*}
(e,h|_y)\inv (p(\beta(h,b))^{-1} x,k) &= ((h|_y)\inv\cdot ((h\cdot y)^{-1}x), (h|_y)\inv|_{(h\cdot y)^{-1}x} k)\\
&= ((h|_y)\inv\cdot ((h|_y \cdot y^{-1})x), (h\inv |_{h\cdot y})|_{(h\cdot y)^{-1}x} k)\\
&= (y\inv (h|_y)\inv|_{h|_y\cdot y\inv}\cdot x, h\inv |_{x} k) \\
&= (y\inv (h\inv|_{h\cdot y})|_{(h\cdot y)^{-1}}\cdot x, h\inv |_{x} k) \\
&= (y\inv (h\inv \cdot x, h\inv |_{x} k) .
\end{align*}
Therefore, we have 
\[(M_h\Pi(b)f)(x,k)=(\Pi(\beta(h,b))M_{h|_y}f)(x,k),\]
which implies that 
  \[
        M_{h}\Pi(b) = \Pi (\beta(h,b))M_{h|_{p(b)}}
  \]
  for all $h\in\cH$ and $b\in \cB$. 
\end{proof}

\begin{theorem}\label{thm.inj} The canonical map $i\colon C^*(\cB)\to C^*(\cB\bowtie_\beta \cH)$ is an injective $*$-homomorphism. 
\end{theorem}

\begin{proof} Denote $\cK=\cG\bowtie\cH$ and $\cC=\cB\bowtie_\beta \cH$.  Let $\pi^{\mathrm{u}}\colon \cB\to C^*(\cB)$ be the universal $*$-representation of $\cB$, where $C^*(\cB)$ is understood as a concrete $C^*$-algebra inside some $\mathbb{B}(H)$. Let $\Pi^u\colon\cB\to\mathbb{B}(\ell^2(\cG\bowtie\cH)\otimes H )$ be its twisted amplification. 

By Lemma \ref{lm.M}, we can construct a unitary representation $M$ of $\cH$ such that $(\Pi^u, M)$ is a covariant representation of $(\cB,\beta)$. Let $L_{\Pi^u,M}$ be its integrated form, as constructed in Theorem~\ref{thm:representations-of-ZS-Fellbdl}. 

Let $i\colon C^*(\cB)\to C^*(\cB\bowtie_\beta\cH)$ be the map from Proposition~\ref{prop.blendG}.  Assume that $i(\sigma)=0$ for some $\sigma\in \Gamma_c(\cG, \cB)$; we have to show that $\sigma=0$. So let $\tau=\sigma\square \sigma^*$ and note that 
$i(\tau)=i(\sigma)i(\sigma^*)=0$, since $i$ is a $*$-homomorphism. As
\[\norm{i(\tau)}=\sup\{\norm{L(i(\tau))}: I\text{-norm decreasing }*\text{-rep.\ }L\text{ of }\Gamma_c(\cK; \cC)\},\]
it follows that $L_{\Pi^u,M}(i(\tau))=0$, since $L_{\Pi^u,M}$ is an $I$-norm decreasing $*$- representation
of $\Gamma_c(\cK; \cC)$. 

For any $\xi\in H$, let $f_\xi:= \delta_{(e,e)}\otimes \xi \in \ell^2(\cK,H)$.
For $\eta\in H$, we have
\[
\inner{ L_{\Pi^u,M}(i(\tau)) f_\xi}{ f_\eta} =
\inner{ \sum_{(x,h)\in \cG\bowtie \cH} \Pi^u(i(\tau)_\cB(x,h)) M_h f_\xi}{ f_\eta}
.
\]
Since $i(\tau)(x,h)=0$ if $h\neq e$ and $i(\tau)(x,h)=(\tau(x),h)$ if $h=e$, we have that
\begin{align*}
\inner{ L_{\Pi^u,M}(i(\tau)) f_\xi}{f_\eta}&= \inner{\sum_{x\in \cG} \Pi^u(\tau(x)) f_\xi}{ f_\eta}\\
&=  \sum_{x\in \cG} \sum_{(y,h)\in\cG\bowtie\cH} \inner{ (\Pi^u(\tau(x)) f_\xi)(y,h)}{ f_\eta(y,h)} \\
&\overset{(*)}{=} \sum_{x\in \cG} \inner{ (\Pi^u(\tau(x)) f_\xi)(e,e)}{ \eta} \\
&= \sum_{x\in \cG} \inner{\pi^u(\tau(x))f_\xi(x\inv,e)}{\eta}.
\end{align*}
Here, $(*)$ follows since $f_\eta(y,h)=0$ for $(y,h)=(e,e)$ and $f_\eta(e,e)=\eta$. Similarly, $f_\xi(x\inv,e)=0$ for $x\neq e$ and $f_\xi(e,e)=\xi$, so we have that 
\[ \inner{\pi^u(\tau(e))\xi}{ \eta}=\inner{ L_{\Pi^u,M}(i(\tau)) f_\xi}{ f_\eta} =0.\]

Since $\pi^u$ is the universal representation, $\pi^u$ is faithful on $\cB$ and thus we have $\tau(e)=0$.  One can compute that $\tau(e)=\sum_{x\in \cG} \sigma(x) \sigma(x)^*$.  This implies that 
\[\tau(e)=\sum_{x\in \cG} \sigma(x) \sigma(x)^*=0,\]
and thus $\sigma(x)=0$ for all $x\in \cG$ and $\sigma=0$. Therefore, $i$ is injective as desired. 
\end{proof}

\bibliographystyle{abbrv}

\end{document}